\documentclass[10pt]{article}
\usepackage{amsmath,amssymb,amsthm,mathrsfs,dsfont}

\usepackage[margin=3cm]{geometry} %宽版时用，配合双倍行距
\usepackage{titlesec,hyperref}

\usepackage{color}
\usepackage{fancyhdr}
\titleformat{\subsection}{\it}{\thesubsection.\enspace}{1pt}{}

\newtheorem{theo}{Theorem}[section]
\newtheorem{lemm}[theo]{Lemma}

\newtheorem{coro}[theo]{Corollary}
\newtheorem{prop}[theo]{Proposition}
\newtheorem{rema}[theo]{Remark}
\numberwithin{equation}{section}

\allowdisplaybreaks %允许公式分页显示

\newcommand\lm{{\lesssim}}
\begin{document}
\title{Uniform vanishing damping limit for the 2D inviscid Oldroyd-B model with
fractional stress tensor diffusion
\hspace{-4mm}
}

\author{Chen $\mbox{Liang}^1$ \footnote{Email: liangch89@mail2.sysu.edu.cn},\quad
	Zhaonan $\mbox{Luo}^1$\footnote{E-mail: luozhn7@mail.sysu.edu.cn  } \quad and\quad
	Zhaoyang $\mbox{Yin}^{1}$\footnote{E-mail: mcsyzy@mail.sysu.edu.cn}\\
	$^1\mbox{School}$ of Science,\\ Shenzhen Campus of Sun Yat-sen University, Shenzhen 518107, China}
\date{}
\maketitle
\hrule

\begin{abstract}
This paper is devoted to the uniform vanishing damping limit of the 2D inviscid  Oldroyd-B model with fractional stress tensor diffusion. Firstly, we find that fractional stress tensor diffusion helps to reduce the global regularity of the 2D Oldroyd-B model with damping coefficient $a\in[0,1]$. By virtue of improved Fourier splitting method, we then prove the optimal time decay rates under the critical regularity for $a=0$. When $a\in (0,1]$, we establish time decay rates that are uniform with respect to $a$. Combining the time decay rate for $a\in [0,1]$ and the time integrability, we obtain the uniform damping vanishing rates for the 2D Oldroyd-B model. Using spectral analysis methods, we finally improve the time decay rates for $\mathrm{tr}\tau$ with 
$a\in (0,1]$, which ensure the sharp uniform damping vanishing  rates of $\mathrm{tr}\tau$.\\
\vspace*{5pt}

\vspace*{5pt}
\end{abstract}

\vspace*{-20pt}

%\phantomsection
%\addcontentsline{toc}{section}{\contentsname}
%添加目录到书签
\tableofcontents
	\section{Introduction}

 The generalized Oldroyd-B models \cite{2015Elgindi,1958Non} can be written as follows: 
	\begin{align}\label{eq1}
		\left\{\begin{array}{l}
			\partial_tu + u\cdot\nabla u+\nabla {\rm P} = {\rm div}~\tau+\nu\Delta u,~~~~{\rm div}~u=0,\\[1ex]
			\partial_t\tau + u\cdot\nabla\tau+a\tau+Q(\nabla u,\tau)-\mu\Delta\tau=\alpha D(u),\\[1ex]
			u|_{t=0}=u_0,~~\tau|_{t=0}=\tau_0. \\[1ex]
		\end{array}\right.
	\end{align}
	In \eqref{eq1}, $u(t,x)$ represents the velocity of the polymeric liquid, $\tau(t,x)$ is the symmetric tensor of constrains and $\rm P$ denotes the pressure. $Q$ is the following bilinear form
	$$Q(\nabla u, \tau)=\tau \Omega(u)-\Omega(u)\tau-b(D(u)\tau+\tau D(u)),~~~~~~~b\in[-1, 1],$$
	with vorticity tensor
	$$
	\Omega(u)=\frac {\nabla u-(\nabla u)^T} {2},
	$$
	and deformation tensor
	$$
	D(u)=\frac {\nabla u+(\nabla u)^T} {2}.
	$$ 
The parameters $a$, $\alpha$, $\mu$ and $\nu$ are nonnegative. The case $\alpha = b = 0$ is known as the co-rotational Oldroyd-B model. The Oldroyd-B model describes the dynamics of
viscoelastic fluids. For more explanations of the classical Oldroyd-B models \eqref{eq1}, one can refer to \cite{1958Non,Oldroyd1950On}.

Now, we review the relevant mathematical results for the Oldroyd-B model \eqref{eq1}.

\textbf{$\bullet$} $a>0,\ \mu=0 $ and $\nu>0$

 C. Guillop'e and J. C. Saut \cite{Guillope1990} first obtained an unique global strong solution in Sobolev spaces. L. Molinet and R. Talhouk \cite{Molinet} improved the above results. For the co-rotational Oldroyd-B model, P. L. Lions and N. Masmoudi \cite{Lions-Masmoudi} proved there exists a global weak solution. The existence of weak solutions remains an open problem for the (\ref{eq1}) with $b\neq0$. J. Y. Chemin and N. Masmoudi \cite{Chemin2001} proved there exists a unique strong solution in critical homogenous Besov space and established the blow-up criterion. Later, Z. Lei, N. Masmoudi and Y. Zhou \cite{LEI2010328} refined the blow-up criterion. Optimal time decay rates in $H^2$ framework
 for the (\ref{eq1})with the small initial data have been proved by M. Hieber, H. Wen, and R. Zi \cite{Hieber2019}. For more well-posedness results about Oldroyd-B model can be found in \cite{Zi,Fang,Wan,HUANG2022456,DEANNA2020108761}.
 
 \textbf{$\bullet$} $a=0,\ \mu=0$ and $\nu>0$

Y. Zhu \cite{ZHU20182039} established the global well-posedness for 3D Oldroyd-B model under the smallness assumption of $\||\nabla|^{-1}(u_0,\tau_0)\|_{H^3}$. Analogous results in $\mathbb{R}^{2}$ have also been established by Y. Chen and Y.Zhu \cite{CHEN2023606}. Q. Chen and X. Hao \cite{Chen2019} proved the global result in Besov spaces.  

\textbf{$\bullet$} $a>0,\ \nu=0$ and $\mu>0$

T. M. Elgindi and F. Rousset \cite{2015Elgindi} firstly established global existence for (\ref{eq1}) in the Sobolev space $H^s$  under the smallness assumption of $\|(u_0,\tau_0)\|_{H^1}+\|(\mathrm{curl}\ u_0,\tau_0)\|_{B_{\infty,1}^{0}}$. T. M. Elgindi and J. Liu \cite{2015Elgindi1} showed the global regularity in the 3D case under the small initial data. It should be noted that the damping term plays a crucial role in \cite{2015Elgindi} and \cite{2015Elgindi1}. For $a>0$, X. Chen, Z. Luo, Z. Yang and C. Yuan \cite{CLYY} proved the global regularity in the 2D case by smallness assumption of $\|(u_0,\tau_0)\|_{L^2}+\|(\nabla u_0,\tau_0)\|_{B_{\infty,1}^{0}}$. For $a\in[0,1]$, they proved the global regularity with $d=2,3$ and the small assumption of $\|(u_0,\tau_0)\|_{H^1}+\|(\nabla u_0,\tau_0)\|_{B_{\infty,1}^{0}}$. 
Furthermore, they obtained that the sharp uniform vanishing damping rate is $a^{\frac{d}{4}}$. 

In this paper, we consider the following Oldroyd-B model  with
fractional stress tensor diffusion: 
	\begin{align}\label{eq2}
\left\{\begin{array}{l}
			\partial_tu + u\cdot\nabla u+\nabla {\rm P} = {\rm div}~\tau,~~~~{\rm div}~u=0,\\[1ex]
			\partial_t\tau + u\cdot\nabla\tau+a\tau+Q(\nabla u,\tau)+(-\Delta)^{\beta}\tau=D(u),\\[1ex]
			u|_{t=0}=u_0,~~\tau|_{t=0}=\tau_0. \\[1ex]
		\end{array}\right.
	\end{align}
where $a\in [0,\infty)$, $\beta>0$ and fractional Laplacian operator $(-\Delta)^{\beta}$ is defined by
	$$
	(-\Delta)^{\beta}\tau=\mathscr{F}^{-1}(|\xi|^{2\beta}\widehat{\tau}),\quad\quad \widehat{\tau}(\xi,t)=\mathscr{F}\tau(\xi,t)=\int_{\mathbb{R}^d}e^{-ix\cdot\xi}\tau(x,t) dx.
	$$
Then, we briefly review the relevant mathematical results for \eqref{eq2}.

\textbf{$\bullet$} $a=0,~\beta\in [\frac{1}{2},1]$

 P. Constantin, J. Wu, J. Zhao and Y. Zhu \cite{P.Constantin} proved global well-posedness to (\ref{eq2}) with the small data in Sobolev space $H^s$ for $s>1+\frac{d}{2}$. J. Wu and J. Zhao \cite{Wu} obtained global well-posedness for the small initial data in the critical Besov spaces. 
 Z. Chen, L, Liu, D, Qin and W, Ye \cite{Chenzhi} proved the analogous results in the $L^p$ framework. 
Sharp time decay rate has been studied by P. Wang, J. Wu, X. Xu and Y. Zhong \cite{Wu1}. they proved under sufficiently 
small initial data $(u_0,\tau_0)\in L^1\cap H^{4+\frac{6}{\beta}}$ with $\beta\in [\frac{1}{2},1)$ and $d=2$, the following decay properties hold:
\begin{align*}
&\|u(t)\|_{L^{2}}\lesssim\varepsilon(1+t)^{-\frac{1}{2\beta}},\quad\|u(t)\|_{L^{\infty}}\lesssim \varepsilon(1+t)^{-\frac{1}{\beta}},\quad\|\nabla u(t )\|_{L^{2}}\lesssim\varepsilon(1+t)^{-\frac{1}{\beta}},\\ 
&\|\nabla u(t)\|_{L^{\infty}}\lesssim\varepsilon(1+t)^{-\frac{3}{2\beta}},\quad\|\mathbb{P}\nabla\cdot\tau(t)\|_{L^ {2}}\lesssim\varepsilon(1+t)^{-\frac{1}{\beta}}. 
\end{align*}
They also obtained the decay properties for 3D cases with $\beta\in [\frac{1}{2},1]$. Z. Luo, W. Luo and Z. Yin \cite{LLY} consider the critical case with $d=2,\beta=1$. Using Fourier splitting method and the time weighted energy estimates, they derived the optimal time decay rate of $\|(u,\tau)\|_{ H^s}$ with $s>2$ under smallness condition.  Furthermore, the optimal time decay rate with $\beta\in [\frac{1}{2},1)$ can be obtained by similar methods. Finally, they got the following time decay rate:
\begin{align*}
\left\|\Lambda^{s}(u, \tau)\right\|_{L^{2}} \lesssim (1+t)^{-\frac{s+1}{2\beta}}.
\end{align*}

	\subsection{The main results}

In this paper, we explore the uniform vanishing damping limit for the 2D inviscid Oldroyd-B model (\ref{eq2}) with $\beta\in [\frac{1}{2},1)$. We first establish global regularity in Sobolev spaces for any $a\in[0,1]$. A precise statement is as follows.
\begin{theo}\label{1theo1}
Let $d=2,\ a\in [0,1],\ \frac{1}{2}<\beta<1$, $s\ge2\beta $ or $\beta=\frac{1}{2}$, $s>1$. Let $(u^a,\tau^a)$ be a strong solution of (\ref{eq2}) with the 
initial data $(u_{0},\tau_{0})\in H^{s}$ and $ \tau_{0}$ is symmetric. There exists a small constant $\delta $ such that if 
\begin{align*}
\|(u_{0},\tau_{0})\|_{H^{s}}\leq\delta,
\end{align*}
then the system (\ref{eq2}) admits a unique global strong solution $(u^a,\tau^a)\in C([0,\infty),H^{s})$. Moreover, we obtain thar for all $t>0$, there holds
\begin{align*}
\quad\frac{d}{dt}\left(\|(u^a,\tau^a)\|_{H^{s}}^{2}+2k\langle-\nabla u^a,\tau^a\rangle_{H^{s-\beta}}\right)+\frac{k}{2}\|\nabla u^a\|_{H^{s-\beta}}^{2}+\|\Lambda^{\beta}\tau^a\|_{H^{s}}^{2}\leq 0,
\end{align*}
where $k$ is a sufficiently small constant.
\end{theo}
\begin{rema}
	Note that the result of critical global regularity in Sobolev spaces fails for $\beta=1$. The phenomenon is unique to the Oldroyd-B equations (\ref{eq2}) with fractional stress tensor diffusion.
\end{rema}
\begin{rema}
J. Wu and J. Zhao \cite{Wu} proved if 
 $$\tau_0\in\dot{B}^{\frac{d}{2}+1-2\beta}_{2, 1}\cap \dot{B}^{\frac{d}{2}}_{2, 1}, u_0\in\dot{B}^{\frac{d}{2}+1-2\beta }_{2, 1}\cap \dot{B}^{\frac{d}{2}+2\beta-1}_{2, 1},$$
 then (\ref{eq2}) has a unique global solution with small initial data. Theorem \ref{1theo1} implies the global well-posedness in critical Sobolev space $H^{2\beta}$ for $d=2$. Note that $H^{2\beta}\hookrightarrow B^{2\beta }_{2, 1}$. 
\end{rema}

We then consider the optimal time decay rates for (\ref{eq2}) with $a=0$, which is crucial for us to analysis the uniform vanishing damping rates. Our result states as follows.
\begin{theo}\label{1theo2}
Let $d=2,\ a=0$. Let $(u^0,\tau^0)$ be a strong solution of (\ref{eq2}) with the initial data  $\left(u_{0}, \tau_{0}\right)$ under the condition in Theorem \ref{1theo1}. If additionally $\left(u_{0}, \tau_{0}\right) \in \dot{B}_{2, \infty}^{-1}$, then there exists  $C>0$  such that for every $t>0$ and $s_1\in[0,s]$,
\begin{align*}
\left\|\Lambda^{s_1}(u^0, \tau^0)\right\|_{L^{2}} \leq C(1+t)^{-\frac{s_1+1}{2 \beta}}.
\end{align*}
Furthermore, suppose that  $\left(u_{0}, \tau_{0}\right) \in H^{s+\beta}$  and  $0<\left|\int_{\mathbb{R}^{2}}\left(u_{0}, \tau_{0}\right) d x\right|$ , then there exists  $C_{\beta} \leq C $ such that
\begin{align*}
\left\|\Lambda^{s_1}(u^0, \tau^0)\right\|_{L^{2}} \geq \frac{C_{\beta}}{2}(1+t)^{-\frac{s_1+1}{2 \beta}}.
\end{align*}
\end{theo}

Then we consider the time decay rate and key integrability for  (\ref{eq2}) with $a\in [0,1]$.

\begin{theo}\label{1theo3}
  Let $d=2,\ a\in [0,1]$, $\frac{1}{2}\le\beta<1$ and $s\ge 1+\beta$. Let $(u^a,\tau^a)$  be a strong solution of (\ref{eq2}) with the initial data $(u_{0},\tau_{0})\in H^s$ under the condition in Theorem \ref{1theo1}. If additionally $\left(u_{0}, \tau_{0}\right) \in \dot{B}_{2, \infty}^{-1}$, then there exists a positive constant $C$  such that
  \begin{align*}
 \|(u^{a},\tau^{a})\|_{L^{2}}+(1+t)^{\frac{1}{2}}\|\nabla(u^{a},\tau^{a})\|_{L^ {2}}\leq C(1+t)^{-\frac{1}{2}}.
  \end{align*}
  If additionally $s>1+\beta$, we have
  \begin{align*}
   \int_{0}^{\infty}\|\nabla (u^{a},\tau^{a})\|_{L^{\infty}}dt\leq C.
  \end{align*}
  \end{theo}

The following result focus on the uniform-in-time vanishing damping rate of (\ref{eq2}).
  \begin{theo}\label{1theo5}
  Let $d=2,\ a\in [0,1]$, $\frac{1}{2}\le\beta<1$ and $s> 1+\beta$. Let $(u^a,\tau^a)$  be a strong solution of (\ref{eq2}) with the initial data $(u_{0},\tau_{0})\in H^s\cap\dot{B}_{2, \infty}^{-1}$ under the condition in Theorem \ref{1theo1}. When $\frac{1}{2}<\beta<1$, the following uniform-in-time vanishing damping rate holds
  \begin{align*}
   \|\nabla^{\alpha}( u^{a}- u^{0},\tau^{a}-\tau^{0})\|_{L^{\infty}\left([0,\infty),L^{2}\right)}\leq Ca ^{\frac{\beta(1+\alpha)}{\alpha\beta+3\beta-1}},
  \end{align*}
 where $\alpha\in[0,1]$. 
  When $\beta=\frac{1}{2}$, we have
  \begin{align*}
  \|\nabla^{\alpha}( u^{a}-u^{0},\tau^{a}-\tau^{0})\|_{L^{\infty}([0,\infty),L^{2})}\leq Ca ^{1-\varepsilon},
  \end{align*}  
for any $0<\varepsilon<1$ and $\alpha\in[0,1]$.
\end{theo}
We finally concerned with the sharp uniform-in-time vanishing damping rate for $\mathrm{tr}\tau^a $.
\begin{theo}\label{1theo6}
Let $d=2,\ a\in [0,1]$, $\frac{1}{2}\le\beta<1$ and $s\ge 1+\beta$. Let $(u^a,\tau^a)$  be a strong solution of (\ref{eq2}) with the initial data $(u_{0},\tau_{0})\in H^s\cap\dot{B}_{2, \infty}^{-1}$ under the condition in Theorem \ref{1theo1}. Taking $\tilde{\tau}^a=\mathrm{tr }\tau^a $, we have
\begin{align*}
\|\tilde{\tau}^{a}\|_{L^2}\le C(1+t)^{-\frac{1}{2\beta}}.
\end{align*}
Moreover, if $\frac{1}{2}<\beta<1$ and $s>1+\beta$, then the following sharp uniform-in-time vanishing damping rate holds
\begin{align}\label{10004}
\|\tilde{\tau}^{a}-\tilde{\tau}^{0}\|_{L^{\infty}\left([0,\infty),L^{2}\right)}\le Ca^{\frac{1}{2\beta}}.
\end{align}
\end{theo}
\begin{rema}
  The uniform-in-time $L^2$ vanishing damping rate (\ref{10004}) is sharp in the sense that we can
find initial data satisfying our conditions such that
\begin{align*}
\|\tilde{\tau}^{a}-\tilde{\tau}^{0}\|_{L^{\infty}\left([0,\infty),L^{2}\right)}\ge Ca^{\frac{1}{2\beta}}.
\end{align*}
One can see Proposition \ref{4prop14} for more details.
\end{rema}

\subsection{Motivations and main ideas}

The well-posedness and large time behavior of (\ref{eq2}) have been widely studied. We firstly consider critical global reguarity in Sobolev space.  Moreover, the uniform vanishing damping limit for (\ref{eq2}) with $\frac{1}{2}\le\beta<1$ has not been discussed. The another aim of this paper is to investigate this problem. 
Then we divide the results of global theory into the following two cases. Let $(u^a,\tau^a)$ be a solution of (\ref{eq2})

\textbf{Case 1: Global regularity and Optimal time decay rates.}

We firstly establish global regularity for (\ref{eq2}). Compared with \cite{P.Constantin,Wu}, we extend the range of initial value corresponding to the global solutions with the case $\frac{1}{2}\le\beta<1$. Thanks to $\frac{1}{2}\le\beta<1$, we observe that when $s=2\beta$,
\begin{align*}
\langle [\Lambda^{2\beta}, u^a\cdot\nabla]u^a,\Lambda^{2\beta} u^a \rangle&\lesssim\ \|\Lambda^{2\beta}u^a\|_{L^{\frac{2}{\beta}}}^{2}\|\nabla u^a\|_{L^{\frac{1}{1-\beta}}}
\lesssim\ \|\Lambda^{\beta+1}u^a\|_{L^2}^{2}\|u^a\|_{H^{2\beta}}.
\end{align*}
Note that the above estimate fails for $\beta=1$. When $\beta=\frac{1}{2}$, we need to dealt with $\|(u^a,\tau^{a})\|_{L^\infty}$, which can be bounded by $\|(u^a,\tau^a)\|_{H^{s}}$ with $s>1=2\beta$. Then we can prove Theorem \ref{1theo1} by virtue of energy estimate method and the bootstrap argument. 

When $a=0$, we add a extra term $\left\langle\tau^0,-\Lambda^{-2+2 \beta} \nabla u^0\right\rangle$ to derive lower order energy dissipation estimate for $u^0$. By virtue of the improved Fourier splitting method, we obtain initial decay rate 
\begin{align*}
\|(u^0,\tau^0)\|_{H^s}\lesssim(1+t)^{-\frac{1}{4\beta}+\frac{1}{4}}.
\end{align*}
Then we can establish the $H^{\beta}$ optimal time decay rate by iteration techniques. For highest order time decay rate, we construct time weighted energy functional and dissipation functional, which help us close the energy estimate. By virtue of Fourier splitting method and time weighted energy estimate, one can derive the upper bound of decay rate for the highest derivative. Finally, we introduce a new weighted energy estimate instead of complex
spectral analysis to prove the lower bound of the decay rate.

\textbf{Case 2: The uniform vanishing damping limit.}

\textbf{(1) Establishing uniform time decay rate for (\ref{eq2}) with $a\in(0,1]$.} 

For $a\in (0,1]$, if we add the same term $\left\langle\tau^a,-\Lambda^{-2+2 \beta} \nabla u^a\right\rangle$, we found
\begin{align*}
&\frac{d}{dt}\left\langle\tau^a,-\Lambda^{-2+2 \beta} \nabla u^a\right\rangle+\frac{1}{2}\|\Lambda^{\beta} u^a\|_{L^{2}}^{2} \\ \nonumber
=&\left\langle G-\Lambda^{2 \beta} \tau^a,-\Lambda^{-2+2 \beta} \nabla u^a\right\rangle-\left\langle \mathbb{P}(F+\operatorname{div} \tau^a),\Lambda^{-2+2 \beta} \operatorname{div} \tau^a\right\rangle-a\left\langle \tau^a,\Lambda^{-2+2 \beta} \nabla u^a\right\rangle. \nonumber
\end{align*}
The first two terms on the right-hand side can be controlled. However, one can see that
\begin{align*}
a\left\langle \tau^a,-\Lambda^{-2+2 \beta} \nabla u^a\right\rangle\lesssim a\|\tau^a\|_{L^2}\|\Lambda^{-2+2 \beta} \nabla u^a\|_{L^2}.
\end{align*}
The term $\|\Lambda^{-2+2 \beta} \nabla u^a\|_{L^2}$  can not be bounded by $\|\nabla u^a\|_{H^{s-\beta}}$. Thus, we fail to get the optimal time decay rate for (\ref{eq2}) with $a\in (0,1]$ by virtue of similar method in Theorem \ref{1theo2}.

However, we can derive initial time decay rate  
\begin{align*}
\|(u^a,\tau^a)\|_{L^2}\lesssim\ln^{-\frac{1}{2}}(e+t).
\end{align*}
By virtue of the improved Fourier splitting method again, one can improve the time decay rate. Finally, we obtain  
 \begin{align*}
 \|(u^{a},\tau^{a})\|_{L^{2}}+(1+t)^{\frac{1}{2}}\|\nabla(u^{a},\tau^{a})\|_{L^ {2}}\leq C(1+t)^{-\frac{1}{2}}.
 \end{align*}
 If additionally $s>1+\beta$, we have
 \begin{align*}
 \int_{0}^{\infty}\|\nabla (u^{a},\tau^{a})\|_{L^{\infty}}dt\leq C.
 \end{align*}

\textbf{(2) Demonstrating the uniform vanishing damping rate.} 

In \cite{CLYY}, we find that the uniform $L^2$ vanishing damping rate depend on the optimal decay rate of $\|(u^a,\tau^a)\|_{H^1}$ and global time integrability for $\|\nabla (u^a,\tau^a)\|_{L^\infty}$. Unlike the case with $\beta=1$, we cannot close  energy estimate in $L^2$. Fortunately, we can achieve the $H^1$ energy estimate. For $\frac{1}{2}<\beta<1$, there holds
\begin{align*}
  \|(u^a-u^0,\tau^a-\tau^0)\|_{H^1}\lesssim a(1+t)^{-\frac{1}{2\beta}}.
\end{align*}
Moreover, using the time decay rate in $H^1$ and global time integrability, for any $\alpha\in [0,1]$, we conclude that
 $$ \|(u^a-u^0,\tau^a-\tau^0)\|_{\dot{H}^{\alpha}}\lesssim  a^{\frac{\beta(1+\alpha)}{\alpha\beta+3\beta-1}}.$$ 
There exists a little difference when dealing with the case with $\beta=\frac{1}{2}$. We infer that
\begin{align*}
   \|(u^{a}- u^{0},\tau^{a}-\tau^{0})\|_{H^{1}}\lesssim a\mathrm{log}(1+t).
\end{align*}
The presence of the logarithm leads to a loss of $a$, then we deduce the uniform vanish damping rate is $a^{1-}$.

\textbf{(3) Demonstrating the optimal uniform vanishing damping rate for $\mathrm{tr}\tau^a$.} 

Thanks to $\mathrm{div}~Du^a=0$, considering the trace of the equation (\ref{eq2}), we infer that $\tilde{\tau}^{a}=\mathrm{tr}\tau^a$ satisfies the fractional heat equation:
\begin{align*}
\partial_{t}\tilde{\tau}^{a}+(-\Delta)^{\beta}\tilde{\tau}^{a}+a\tilde{\tau}^{a}=Q_{1}(u ^{a},\nabla\tau^{a})+Q_{2}(\nabla u^{a},\tau^{a}).
\end{align*}
where $Q_1$ and $Q_2$ are quadratic terms. By virtue of spectral analysis, we know that the low frequencies of $\tilde{\tau}^{a}$ behave as fractional diffusion, the high frequencies of $\tilde{\tau}^{a}$ should be damped. Then we obtain the $L^2$ optimal time decay rate of $\tilde{\tau}^{a}$ with $\frac{1}{2}\le\beta<1$. Then we refine the $L^2$ uniform vanishing damping rate for $\tilde{\tau}^{a}$ with $\frac{1}{2}<\beta<1$.

Finally, we construct initial data $(u_0,\tau_0)$ satisfying the conditions of Theorem \ref{1theo3} such that the lower bound holds
\begin{align*}
\|\tilde{\tau}^{a}-\tilde{\tau}^{0}\|_{L^2}\gtrsim a^{\frac{1}{2\beta}},
\end{align*}
for $\frac{1}{2}<\beta<1$. This means that the improved vanishing damping rate for $\mathrm{tr}\tau^a$ is optimal. 

Several related questions are remaining open after this work. For the endpoint case with $\beta=\frac{1}{2}$ and $H^1$ initial data, the global regularity for (\ref{eq2}) remains unresolved. The corresponding linearized system of (\ref{eq2}) for $\beta=\frac{1}{2}$  is as follows:
    	\begin{align*}
		\left\{\begin{array}{l}
			\partial_tu + \nabla {\rm P} = {\rm div}~\tau,~~~~{\rm div}~u=0,\\[1ex]
			\partial_t\tau + \Lambda\tau= D(u).
		\end{array}\right.
	\end{align*}
The type of (\ref{eq2}) changes from parabolic equation to dispersion equation. Hence, the case with $\beta=\frac{1}{2}$ is more challenging than the case with $\frac{1}{2}<\beta<1$.

{\rm\textbf{Structure of the paper:}}~~The rest of this paper is divided into four sections. In Section 2, we give some preliminaries which will be used in the sequel. In Section 3, we prove the global regularity for (\ref{eq2}). In Section 4, we will focus on optimal time decay rate for the (\ref{eq2}) with $a=0$. In Section 5, we study the uniform vanishing damping limit for (\ref{eq2}).

	\section{Preliminaries}
\par
	In this section, we introduce some notations and useful lemmas which will be used in the sequel.
	
	We agree that $f\lm g$ represents $f\leq Cg$ with a constant $C>0$. The symbol $\widehat{f}=\mathscr{F}(f)$ stands for the Fourier transform of $f$. Denote $\mathscr{F}^{-1}(f)$ the inverse Fourier transform of $f$.
	
	The Littlewood-Paley decomposition theory and Besov spaces are given as follows.
	\begin{lemm}\cite{Bahouri2011}\label{lemma1}
		Let $\mathscr{C}$ be the annulus $\{\xi\in\mathbb{R}^d:\frac 3 4\leq|\xi|\leq\frac 8 3\}$. There exists a radial functions $\varphi$, valued in the interval $[0,1]$, belonging to $\mathscr{D}(\mathscr{C})$, and such that
		%$$ \forall\xi\in\mathbb{R}^d,\ \chi(\xi)+\sum_{j\geq 0}\varphi(2^{-j}\xi)=1, $$
		$$ \forall\xi\in\mathbb{R}^2\backslash\{0\},\ \sum_{j\in\mathbb{Z}}\varphi(2^{-j}\xi)=1,~~~ $$
		$$ |j-j'|\geq 2\Rightarrow\mathrm{Supp}\ \varphi(2^{-j}\cdot)\cap \mathrm{Supp}\ \varphi(2^{-j'}\cdot)=\emptyset, $$
		%$$ ~~j\geq 1\Rightarrow\mathrm{Supp}\ \chi(\cdot)\cap \mathrm{Supp}\ \varphi(2^{-j}\cdot)=\emptyset. $$
		%The set $\widetilde{\mathscr{C}}=B(0,\frac 2 3)+\mathscr{C}$ is an annulus, then
		%$$ |j-j'|\geq 5\Rightarrow 2^{j}\mathscr{C}\cap 2^{j'}\widetilde{\mathscr{C}}=\emptyset. $$
		Moreover, we have
		%$$ ~~\forall\xi\in\mathbb{R}^d,\ \frac 1 2\leq\chi^2(\xi)+\sum_{j\geq 0}\varphi^2(2^{-j}\xi)\leq 1, $$
		$$ \forall\xi\in\mathbb{R}^2\backslash\{0\},\ \frac 1 2\leq\sum_{j\in\mathbb{Z}}\varphi^2(2^{-j}\xi)\leq 1.~~ $$
	\end{lemm}
			
		Let $u$ be a tempered distribution in $\mathcal{S}_{h}'(\mathbb{R}^2)$. For all $j\in\mathbb{Z}$, define
		%$$
		%\Delta_j u=0\,\ \text{if}\,\ j\leq -2,\quad
		%\Delta_{-1} u=\mathscr{F}^{-1}(\chi\mathscr{F}u),$$
		%$$\Delta_j u=\mathscr{F}^{-1}(\varphi(2^{-j}\cdot)\mathscr{F}u)\,\ \text{if}\,\ j\geq 0,\quad
		%S_j u=\sum_{j'<j}\Delta_{j'}u.
		%$$
		The homogeneous operators are defined by
		$$\dot{\Delta}_j u=\mathscr{F}^{-1}(\varphi(2^{-j}\cdot)\mathscr{F}u).$$
Then the Littlewood-Paley decomposition is given as follows:
$$
u=\sum_{j\in\mathbb{Z}}\dot{\Delta}_j u\quad  \mathrm{in}\quad  \mathcal{S}_{h}'(\mathbb{R}^2). 
$$		
		Let $s\in\mathbb{R}$ and $(p,r)\in[1,\infty]^2$. The homogeneous Besov space $\dot{B}^s_{p,r}$ is given as follows
	$$ \dot{B}^s_{p,r}=\left\{u\in \mathcal{S}_{h}'(\mathbb{R}^2):\|u\|_{\dot{B}^s_{p,r}}=\Big\|(2^{js}\|\dot{\Delta}_j u\|_{L^p})_j \Big\|_{l^r(\mathbb{Z})}<\infty\right\}.$$

We introduce the Gagliardo-Nirenberg inequality for
$d = 2$.
\begin{lemm}\label{lemma3}\cite{1959On}
For $d=2$, $p\in[2,+\infty)$ and $0\leq s,s_{1}\leq s_{2}$, there holds
\begin{align*}
\|\Lambda^{s}f\|_{L^{p}}\leq C\|\Lambda^{s_{1}}f\|_{L^{2}}^{1- \theta}\|\Lambda^{s_{2}}f\|_{L^{2}}^{\theta},
\end{align*}
where $0\leq\theta\leq 1$ and satisfies  
$$s+1-\frac{2}{p}=(1-\theta)s_{1}+\theta s_{2}.$$ 
Note that we also require that $0<\theta<1, 0\leq s_{1}\leq s$, when $ p=\infty$.
\end{lemm}

The following commutator estimate and product estimate are useful to control convective term and $Q(\nabla u,\tau)$.
\begin{lemm}\label{lemma4}\cite{kato1}
Assume that $s>0$, $p, p_1, p_4 \in (1,\infty)$ and $\frac{1}{p}=\frac{1}{p_1}+\frac{1}{p_2}=\frac{1}{p_3}+\frac{1}{p_4}$, then we obtain
\begin{align*}
\|[\Lambda^{s},f]g\|_{L^{p}}\leq C\left(\|\Lambda^{s}f\|_{L^{p_{1}}}\|g\|_{L^{p_{2} }}+\|\nabla f\|_{L^{p_{3}}}\|\Lambda^{s-1}g\|_{L^{p_{4}}}\right).
\end{align*}
\end{lemm}
\begin{lemm}\label{lemma5}\cite{kato1}
Assume that $s>0$, $p, p_2, p_4 \in (1,\infty)$ and $\frac{1}{p}=\frac{1}{p_1}+\frac{1}{p_2}=\frac{1}{p_3}+\frac{1}{p_4}$, then we obtain
\begin{align*}
\|\Lambda^{s}\left(fg\right)\|_{L^{p}} \leq C\left(\|f \|_{L^{p_{1}}}\|\Lambda^{s}g\|_{L^{p_{2}}} + \|g\|_{L^{p_ {3}}}\|\Lambda^{s}f\|_{L^{p_{4}}}\right).
\end{align*}
\end{lemm}

\section{Critical global regularity}
In this section, we shall prove the critical global regularity for (\ref{eq2}).
 
 \textbf{Proof of Theorem \ref{1theo1}:}\\
We only prove the critical case with $s=2\beta$ and $\beta\in(\frac 1 2,1)$. Assume that $(u^a,\tau^a)$ be a local strong solution of (\ref{eq2}).
Taking the $L^2$ inner product to (\ref{eq2}) with $(u,\tau)$ and integrating by parts, we get
\begin{align*}
\quad\frac{1}{2}\frac{d}{dt}\|(u^a,\tau^a)\|_{L^{2}}^{2}+\|\Lambda^{\beta}\tau^a\|_{L^{2}}^{2}+a\|\tau^a\|_{L^{2}}^{2}=-\langle Q(\nabla u^a,\tau^a), \tau^a\rangle
\end{align*}
where we use
\begin{align*}
\langle u^a,\mathrm{div}\tau^a\rangle+\langle Du^a,\tau^a\rangle=0.
\end{align*}
Then we deduce that 
\begin{align*}
-\langle Q(\nabla u^a,\tau^a), \tau^a\rangle\lesssim \ \|\nabla u^a\|_{L^2}\|\tau^a\|^{2}_{L^4}
\lesssim \ \|\nabla u^a\|_{L^2}\|\tau^a\|_{L^2}\|\nabla \tau^a\|_{L^2}.
\end{align*}
Therefore,
\begin{align}\label{211}
\quad\frac{1}{2}\frac{d}{dt}\|(u^a,\tau^a)\|_{L^{2}}^{2}+\|\Lambda^{\beta}\tau^a\|_{L^{2}}^{2}+a\|\tau^a\|_{L^{2}}^{2}\lesssim\|\nabla u^a\|_{L^2}\|\tau^a\|_{L^2}\|\nabla \tau^a\|_{L^2}.
\end{align}
Applying $\Lambda^{2\beta}$ to (\ref{eq2}), taking the $L^2$ inner product to (\ref{eq2}) with $(\Lambda^{2\beta} u^a,\Lambda^{2\beta} \tau^a)$ and integrating by parts, we obtain
\begin{align*}
&\quad\frac{1}{2}\frac{d}{dt}\|(\Lambda^{2\beta} u^a,\Lambda^{2\beta} \tau^a)\|_{L^{2}}^{2}+\|\Lambda^{3\beta}\tau^a\|_{L^{2}}^{2}+a\|\Lambda^{2\beta}\tau^a\|_{L^{2}}^{2}\\\nonumber
&=-\langle (\Lambda^{2\beta} (u^a\cdot \nabla u^a),\Lambda^{2\beta} u^a \rangle-\langle (\Lambda^{2\beta} (u^a\cdot \nabla \tau^a),\Lambda^{2\beta} \tau^a\rangle
-\langle \Lambda^{2\beta} Q(\nabla u^a,\tau^a), \Lambda^{2\beta}\tau^a\rangle \\\nonumber
&\triangleq\sum_{i=1}^{3}I_i \nonumber.
\end{align*}
Thanks to the fact $\mathrm{div }u^a=0$ and Lemma \ref{lemma4}, we observe 
\begin{align*}
|I_1|=|\langle [\Lambda^{2\beta}, u^a\cdot\nabla]u^a,\Lambda^{2\beta} u^a \rangle|&\lesssim\ \|\Lambda^{2\beta}u^a\|_{L^{\frac{2}{\beta}}}^{2}\|\nabla u^a\|_{L^{\frac{1}{1-\beta}}}\\
&\lesssim\ \|\Lambda^{\beta+1}u^a\|_{L^2}^{2}\|u^a\|_{H^{2\beta}}.
\end{align*}
By virtue of Lemma \ref{lemma5} and $\frac{1}{2}<\beta<1$, we obtain
\begin{align*}
|I_2| \lesssim\ & \|\Lambda^{3\beta}\tau^a\|_{L^2}\left(\|\Lambda^{\beta+1}u^a\|_{L^2}\|\tau\|_{L^\infty}+\|\Lambda^{\beta+1}\tau^a\|_{L^2}\|u\|_{L^\infty}\right).
\end{align*}
We deduce from Lemma \ref{lemma5} that
\begin{align*}
|I_3|\lesssim &\ \|\Lambda^{3\beta}\tau^a\|_{L^2}\|\Lambda^{\beta}Q(\nabla u^a,\tau^a)\|_{L^2}\\
\lesssim &\ \|\Lambda^{3\beta}\tau^a\|_{L^2}\left(\|\Lambda^{\beta+1}u^a\|_{L^2}\|\tau^a\|_{L^\infty}+\|\Lambda^{\beta}\tau^a\|_{L^{\frac{2}{\beta}}}\|\nabla u\|_{L^{\frac{2}{1-\beta}}}\right)\\
\lesssim &\ \|\Lambda^{3\beta}\tau^a\|_{L^2}\left(\|\Lambda^{\beta+1}u^a\|_{L^2}\|\tau^a\|_{L^\infty}+\|\Lambda^{\beta+1}u^a\|_{L^{2}}\|\Lambda\tau^a\|_{L^{2}}\right).
\end{align*}
Combining the estimates from $I_1$ to $I_3$, we have
\begin{align}\label{212}
&\quad\frac{1}{2}\frac{d}{dt}\|(\Lambda^{2\beta} u^a,\Lambda^{2\beta} \tau^a)\|_{L^{2}}^{2}+\|\Lambda^{3\beta}\tau^a\|_{L^{2}}^{2}+a\|\Lambda^{2\beta}\tau^a\|_{L^{2}}^{2}\\\nonumber
&\lesssim\left(\|u^a\|_{H^{2\beta}}+\|\tau^a\|_{H^{2\beta}}\right)\left(\|\Lambda^{\beta+1} u^a\|_{L^{2}}^{2}+\|\Lambda^{\beta+1}\tau^a\|_{H^{2\beta-1}}^{2}\right).
\end{align}
Now we consider the inner product for $u^a$ and $\mathrm{div}\tau^a$ to generate the dissipation of $u^a$. We choose $k>0$, which will be determined later. According to (\ref{eq2}), one can deduce that
\begin{align}\label{213}
  &\frac{d}{dt}\langle u^a,k\mathrm{div}\tau^a\rangle+\frac{k}{2}\|\nabla u^a\|_{L^{2}}^{2}\\\nonumber
  =&-k\langle(u^a\cdot\nabla u^a),\mathbb{P}\mathrm{div}\tau^a\rangle-k\langle\mathrm{div}(u^a\cdot\nabla\tau^a),u\rangle
  -k\langle\mathrm{div} Q(\nabla u^a,\tau^a),u^a\rangle\\ \nonumber
  &-k\langle(-\Delta)^\beta\mathrm{div}\tau^a,u^a\rangle-ak\langle\mathrm{div}\tau^a,u^a\rangle
  +k\langle\mathbb{P}\mathrm{div}\tau^a, \mathrm{div}\tau^a\rangle\\\nonumber
  \lesssim&\ \|u^a\|_{L^\infty}\|\nabla u^a\|_{L^2}\|\nabla \tau^a\|_{L^2}+\|\tau^a\|_{L^\infty}\|\nabla u^a\|_{L^2}^{2}\\\nonumber
  &+\|\Lambda^{\beta}\tau^a\|_{L^2}\|\Lambda^{\beta+1} u^a\|_{L^2}^{2}+a\|\tau^a\|_{L^2}\|\nabla u^a\|_{L^2}^{2}+\|\nabla\tau^a\|_{L^{2}}^{2}\\\nonumber
  \le&\ C(\|u^a\|_{H^{2\beta}}+\|\tau^a\|_{H^{2\beta}})\left(\|\nabla u^a\|_{H^{\beta}}^{2}+\|\Lambda^{\beta}\tau^a\|_{H^{2\beta}}^{2}\right)\\\nonumber
  &+\frac{1}{16}\|\Lambda^{\beta}\tau^a\|_{H^{2\beta}}^2+4k^{2}\|\nabla u^a\|_{H^{\beta}}^{2}+\frac{a^{2}}{16}\|\tau^a\|_{H^{2\beta}}^{2}+4k^{2}\|\nabla u^a\|_{H^{\beta}}^{2}+k\|\Lambda^{\beta}\tau^a\|_{H^{2\beta}}^{2}.
\end{align}
Along the same line, we find
\begin{align*}
  &\frac{d}{dt}\langle \Lambda^{\beta}u^a,k\Lambda^{\beta}\mathrm{div}\tau^a\rangle_{L^{2}}+\frac{k}{2}\|\Lambda^{\beta}\nabla u^a\|_{L^{2}}^{2} \\
  =&-k\langle\Lambda^{\beta}(u^a\cdot\nabla u^a),\Lambda^{\beta}\mathbb{P}\mathrm{div}\tau^a\rangle-k\langle\Lambda^{\beta}\mathrm{div}(u^a\cdot\nabla\tau^a),\Lambda^{\beta}u\rangle
  -k\langle\Lambda^{\beta}\mathrm{div} Q(\nabla u^a,\tau^a),\Lambda^{\beta}u^a\rangle\\
  &-k\langle\Lambda^{\beta}(-\Delta)^\beta\mathrm{div}\tau^a,\Lambda^{\beta}u^a\rangle-ak\langle\Lambda^{\beta}\mathrm{div}\tau^a,\Lambda^{\beta}u^a\rangle
  +k\langle\Lambda^{\beta}\mathbb{P}\mathrm{div}\tau^a, \Lambda^{\beta}\mathrm{div}\tau^a\rangle\\
  \triangleq&\sum_{i=4}^{9}I_i .
\end{align*}
Firstly, we infer that
\begin{align*}
|I_4|\lesssim\|\Lambda^{\beta+1}u^a\|_{L^2}\|\Lambda^{\beta+1}\tau^a\|_{L^2}\|u^a\|_{L^\infty}.
\end{align*}
and 
\begin{align*}
|I_5|\lesssim \|\Lambda^{\beta+1}u^a\|_{L^2}\left(\|\Lambda^{\beta+1}u^a\|_{L^2}\|\tau^a\|_{L^\infty}+\|\Lambda^{\beta+1}\tau^a\|_{L^2}\|u^a\|_{L^\infty}\right).
\end{align*}
By a similar derivation of $I_3$, we deduce that
\begin{align*}
|I_6|&\lesssim \|\Lambda^{\beta+1}u^a\|_{L^2}\left(\|\Lambda^{\beta+1}u^a\|_{L^2}\|\tau^a\|_{L^\infty}+\|\Lambda^{\beta}\tau^a\|_{L^{\frac{2}{\beta}}}\|\nabla u\|_{L^{\frac{2}{1-\beta}}}\right)\\
&\lesssim \|\Lambda^{\beta+1}u^a\|_{L^2}\left(\|\Lambda^{\beta+1}u^a\|_{L^2}\|\tau^a\|_{L^\infty}+\|\Lambda \tau^a\|_{L^{2}}\|\Lambda^{1+\beta} u\|_{L^{2}}\right).
\end{align*}
Using Young's inequality, we have
\begin{align*}
|I_7|\le &k\|\Lambda^{3\beta}\tau^a\|_{L^2}\|\Lambda^{\beta+1} u^a\|_{L^2}
\le Ck\|\Lambda^{3\beta}\tau^a\|_{L^2}^2+\frac{k}{100}\|\Lambda^{\beta+1} u^a\|_{L^2}^{2},\\
|I_8|\le &\frac{a}{16}\|\tau^a\|_{H^{2\beta}}^{2}+4ak^{2}\|\nabla u^a\|_{H^{\beta}}^{2}.
\end{align*}
Due to properties of Leray projection operator, we get
\begin{align*}
|I_9| \lesssim k\|\Lambda^{1+\beta}\tau^a\|_{L^{2}}^{2}.
\end{align*}
Substituting the estimates from $I_4$ to $I_9$, we yield 
\begin{align}\label{214}
  &\frac{d}{dt}\langle \Lambda^{\beta} u^a,k\Lambda^{\beta}\mathrm{div}\tau^a\rangle_{L^{2}}+\frac{k}{2}\|\Lambda^{\beta}\nabla u^a\|_{L^{2}}^{2} \\\nonumber
  \le &\ (\|u^a\|_{H^{2\beta}}+\|\tau^a\|_{H^{2\beta}})(\|\Lambda^{1+\beta} u^a\|_{L^{2}}^2+\|\Lambda^{1+\beta}\tau^a\|_{L^{2}}^2)\\
  &+Ck\|\Lambda^{3\beta}\tau^a\|_{L^{2}}^2+\frac{k}{100}\|\Lambda^{1+\beta} u^a\|_{L^{2}}^{2}+\frac{a}{16}\|\tau^a\|_{H^{2\beta}}^{2}+4ak^{2}\|\nabla u^a\|_{H^{\beta}}^{2}+k\|\Lambda^{1+\beta}\tau^a\|_{L^{2}}^{2}.\nonumber
\end{align}
Combining (\ref{211})-(\ref{214}) and choosing $k$ small enough, we conclude that 
\begin{align}\label{215}
&\frac{1}{2}\frac{d}{dt}\left(\|(u^a,\tau^a)\|_{H^{2\beta}}^{2}+2k\langle-\nabla u^a,\tau^a\rangle_{H^{\beta}}\right)+\frac{k}{4}\|\nabla u^a\|_{H^{\beta}}^{2}+\frac{1}{2}\|\Lambda^{\beta}\tau^a\|_{H^{2\beta}}^{2}+\frac{a}{2}\|\tau^a\|_{H^{2\beta}}^{2}\\ \nonumber
\lesssim & \left(\|u^a\|_{H^{2\beta}}+\|\tau^a\|_{H^{2\beta}})(\|\nabla u^a\|_{H^{\beta}}^{2}+\|\Lambda^{\beta}\tau^a\|_{H^{2\beta}}\right).\nonumber
\end{align}
Integrating (\ref{215}) in time on $[0,t]$, then we deduce that
\begin{align*}
&\sup_{t}\|u^a\|_{H^{{2\beta}}}^{2}+\sup_{t}\|\tau^a\|_{H^{{2\beta}}}^{2}+\int_{0}^{t}(\|\Lambda^{\beta}\tau^a\|_{H^{{2\beta}}}^{2}+\frac{k}{2}\|\nabla u^a\|_{H^{\beta}}^{2})dt^{\prime} \\
\lesssim& \|u_0\|_{H^{2\beta}}^2+\|\tau_0\|_{H^{2\beta}}^2+(\sup_t\|u^a\|_{H^{2\beta}}+\sup_t\|\tau^a\|_{H^{2\beta}})\int_0^t(\|\Lambda^\beta\tau^a\|_{H^{2\beta}}^2+\|\nabla u^a\|_{H^{\beta}}^2)dt^\prime.
\end{align*}
There exists $\delta>0$, when $\|(u_0,\tau_0)\|_{H^{2\beta}}<\delta$, we conclude by the bootstrap arguments that
\begin{align}\label{217}
\|u^a\|_{H^{{2\beta}}}+\|\tau^a\|_{H^{{2\beta}}}\lesssim\delta.
\end{align}
Combining (\ref{217}) and (\ref{215}) and taking $\delta$ small enough, we get
\begin{align*}
\frac{1}{2}\frac{d}{dt}\left(\|(u^a,\tau^a)\|_{H^{2\beta}}^{2}+2k\langle-\nabla u^a,\tau^a\rangle_{H^{\beta}}\right)+\frac{k}{4}\|\nabla u^a\|_{H^{\beta}}^{2}+\frac{1}{2}\|\Lambda^{\beta}\tau^a\|_{H^{2\beta}}^{2}\lesssim  0.
\end{align*}
Then we achieve the conclusion. \hfill$\Box$

	\section{Optimal time decay rate with $a=0$}
	\par
	In this section, we investigate the optimal time decay rate for (\ref{eq2}) with $a=0$. For the sake of convenience, we denote $u^0=u,\ \tau^0=\tau$ in this section. Moreover, we agree that all occurrences of $\delta$ and $C_k$ denote their positive powers throughout this paper. We will present the proof of Theorem \ref{1theo2}. 

\subsection{The upper bound}

We now consider the optimal time decay of (\ref{eq2}) based on Schonbek's \cite{Schonbek1985,Schonbek1991} strategy . According to \cite{LLY}, in order to prove optimal decay rate for (\ref{eq2}), we add a term $\left\langle\tau,-\Lambda^{-2+2 \beta} \nabla u\right\rangle$ to obtain lower energy dissipation estimate for $u$. We introduce the following energy and dissipation functionals for $(u,\tau)$:
\begin{align*}
E_{0}&=\left\|(u, \tau)\right\|_{H^{s}}^{2}+2k\left(\left\langle-\nabla  u, \tau\right\rangle_{H^{s-\beta}}+\langle\tau,-\Lambda^{-2+2\beta}\nabla u \rangle\right),\\
D_{0}&=\frac{k}{2}\left\|\Lambda^{\beta}u\right\|_{H^{s+1-2\beta}}^{2}+\left\|\Lambda^{\beta} \tau\right\|_{H^{s}}^{2},
\end{align*}
where $k$ is a sufficiently small constant.
Then, we get the following proposition.
\begin{prop}{\label{3prop1}}
Let $a=0$. Let $(u,\tau)$  be a strong solution of (\ref{eq2})  with the initial data $(u_{0},\tau_{0})$ under the condition in Theorem \ref{1theo1}. Then we have
\begin{align*}
\frac{d}{dt}E_{0}+D_{0}\le 0.
\end{align*}
\end{prop}
\begin{proof}
A simple calculation ensures that
\begin{align*}
&\quad\frac{d}{dt}\left\langle\tau,-\Lambda^{-2+2 \beta} \nabla u\right\rangle+\frac{1}{2}\left\|\Lambda^{\beta} u\right\|_{L^{2}}^{2} \\ \nonumber
&=\left\langle G,-\Lambda^{-2+2 \beta} \nabla u\right\rangle+\left\langle -\Lambda^{2 \beta} \tau,-\Lambda^{-2+2 \beta} \nabla u\right\rangle-\left\langle \mathbb{P}F,\Lambda^{-2+2 \beta} \operatorname{div} \tau\right\rangle-\left\langle \mathbb{P}\operatorname{div} \tau,\Lambda^{-2+2 \beta} \operatorname{div} \tau\right\rangle\\ \nonumber
&=J_1+J_2+J_3+J_4,
\end{align*}
where $G=-Q(\nabla u,\tau)-u\cdot\nabla\tau$ and $F=-u\cdot\nabla u$. One can yield the following estimates
\begin{align*}
J_1
\lesssim\  & \|Q(\nabla u,\tau)\|_{L^{\frac{2}{2-\beta}}}\|\Lambda^{2\beta-1}u\|_{L^\frac{2}{\beta}}+\|u\cdot\nabla\tau\|_{L^{\frac{2}{2-\beta}}}\|\Lambda^{2\beta-1}u\|_{L^\frac{2}{\beta}}\\
\lesssim\  & \|\nabla u\|_{L^2}\|\tau\|_{L^\frac{2}{1-\beta}}\|\Lambda^{\beta}u\|_{L^2}+\|\nabla \tau\|_{L^2}\|u\|_{L^\frac{2}{1-\beta}}\|\Lambda^{\beta}u\|_{L^2}\\
\lesssim\  & \|\nabla u\|_{L^2}\|\Lambda^{\beta}\tau\|_{L^{2}}\|\Lambda^{\beta}u\|_{L^2}+\|\nabla \tau\|_{L^2}\|\Lambda^{\beta}u\|_{L^{2}}\|\Lambda^{\beta}u\|_{L^2},
\end{align*}
and
\begin{align*}
J_2\lesssim \|\Lambda^{3\beta-1}\tau\|_{L^2}\|\Lambda^{\beta}u\|_{L^2}.
\end{align*}
Moreover, we find
\begin{align*}
J_3\lesssim\  \|u\cdot\nabla u\|_{L^{\frac{2}{2-\beta}}}\|\Lambda^{2\beta-1}\tau\|_{L^\frac{2}{\beta}}
\lesssim\   \|\nabla u\|_{L^2}\|\Lambda^{\beta}u\|_{L^{2}}\|\Lambda^{\beta}\tau\|_{L^2},~~~J_4\lesssim  \|\Lambda^{\beta}\tau\|_{L^2}^2.
\end{align*}
Collecting the bounds for $J_1$ through $J_4$, we find 
\begin{align}\label{311}
\frac{d}{dt}\left\langle\tau,-\Lambda^{-2+2 \beta} \nabla u\right\rangle+\frac{1}{2}\left\|\Lambda^{\beta} u\right\|_{L^{2}}^{2} \le C\left(\delta+\frac{1}{100}\right) \|\Lambda^{\beta}u\|_{L^2}^{2}+C(\|\Lambda^{3\beta-1}\tau\|_{L^2}^{2}+\|\Lambda^{\beta}\tau\|_{L^2}^{2}).
\end{align}
By (\ref{311}) and Theorem \ref{1theo1}, we finish the proof.
\end{proof}

Next we establish the $L^2$ time decay rate in following proposition.
\begin{prop}{\label{3prop2}}
  Under the same conditions as in Proposition \ref{3prop1}, if additionally $(u_0,\tau_0)\in \dot B_{2,\infty}^{-1}$, then there exists $C>0$ such that for every $t>0$, there holds
\begin{align*}
\|\Lambda^{s_1}(u,\tau)\|_{H^{s-s_{1}}}\le C(1+t)^{-\frac{1+s_{1}}{2\beta}},
\end{align*}
where $0\le s_{1}\le\beta$.
\end{prop}
\begin{proof}
Since we focus on the long-time behavior of solutions, $t$ will be taken to be sufficiently large. We only consider the  critical case for $s=2\beta$. We firstly split the proof of into the following four steps.

\textbf{Step 1: inital 0-order time decay rate}

  Denote $S_{1}(t)=\left\{\xi:|\xi|^{2\beta}\le C_{2}(1+t)^{-1}\right\}$, $C_{2}$ is large enough. By Proposition \ref{3prop1}, we get
  \begin{align}\label{321}
  \frac{d}{dt}E_{0}(t)+\frac{C_{2}}{1+t}\left(\frac{k}{2}\|u\|_{H^{2\beta}}^{2}+\| \tau\|_{H^{2\beta}}^{2}\right)\leq \frac{C}{1+t}{\int_{S_{1}(t)}|\widehat{u}(\xi)|^{2}+|\widehat \tau}(\xi)|^{2}d\xi.
  \end{align}  
Applying
Fourier transform to (\ref{eq2}), we have
\begin{align*}
\left\{\begin{array}{l}\widehat{u}^{j}_{t}+i\xi_{j}\widehat{P}-i\xi_{k}\widehat{\tau}^{ jk}=\widehat{F}^{j},\\ [1ex] \widehat{\tau}^{jk}_{t}+|\xi|^{2\beta}\widehat{\tau}^{jk}-\frac{i}{2}(\xi_{k}\widehat{u}^{j}+ \xi_{j}\widehat{u}^{k})=\widehat{G}^{jk}.\end{array}\right.
\end{align*}
Then we duduce
\begin{align}\label{323}
\frac{1}{2}\frac{d}{dt}(| \widehat{u}|^{2}+|\widehat{\tau}|^{2})+|\xi|^{2\beta}|\widehat{\tau}|^{2}=\mathcal{R}e[\widehat{F} \cdot\bar{\widehat{u}}]+\mathcal{R}e[\hat{G}\bar{\widehat{\tau}}].
\end{align}
Integrating (\ref{323}) in time on $[0,t]$, we get
\begin{align}\label{324}
|\widehat{u}|^{2}+|\widehat{\tau}|^{2}\lesssim (|\widehat{u}_{0}|^{2}+|\widehat{\tau}_{0}|^{2})+\int_{0}^{t}|\widehat{F}\cdot\bar{\widehat{u}}|+|\widehat{G}\cdot\bar{\widehat{\tau}}|dt^{\prime}.
\end{align}
Integrating (\ref{324}) over $S_{1}(t)$ with $\xi$, then we derive
\begin{align}\label{325}
\int_{S_{1}(t)}|\widehat{u}|^{2}+|\widehat{\tau}|^{2}d\xi\lesssim \int_{S_{1}(t)}|\widehat{u}_{0}|^{2}+|\widehat{\tau}_{0}|^{2}d\xi+\int_{S_{1}(t)}\int_{0}^{t}|\widehat{F}\cdot\bar{\widehat{u}}|+|\widehat{G}\cdot\bar{\widehat{\tau}}|dt^{\prime}d\xi.   
\end{align}
Due to the fact $E(0)<\infty$ and $(u_0,\tau_0)\in \dot B_{2,\infty}^{-1}$, we find
\begin{align}\label{326}
\int_{S_{1}(t)}(|\hat{u}_{0}|^{2}+|\hat{\tau}_{0}|^{2})d\xi &\lesssim\sum_{j\leq\log_{2}\left[\frac{4}{3}C_{2}^{\frac{1}{2\beta}}(1+t)^{-\frac{1}{2\beta}}\right]}\int_{\mathbb{R}^{2}}2\varphi^{2}(2^{-j}\xi)\left(|\widehat{u}_{0}|^{2}+|\widehat{\tau}_{0}|^{2}\right)d\xi \\ \nonumber
 & \lesssim\sum_{j\leq\log_{2}\left[\frac{4}{3}C_{2}^{\frac{1}{2\beta}}(1+t)^{-\frac{1}{2\beta}}\right]}\left(\|\dot{\Delta}_{j}u_{0}\|_{L^{2}}^{2}+\|\dot{\Delta}_{j}\tau_{0}\|_{L^{2}}^{2}\right) \\ \nonumber
 & \lesssim (1+t)^{-\frac{1}{\beta}}\|(u_{0},\tau_{0})\|_{\dot{B}_{2,\infty}^{-1}}^{2}.\nonumber
\end{align}
By Minkowski's inequality, we arrive at
\begin{align}\label{327}
\int_{S_{1}(t)}\int_{0}^{t}|\widehat{F}\cdot\bar{\widehat{u}}|+|\widehat{G}\cdot\bar{\widehat{\tau}}|dt^{\prime}d\xi=& \int_{0}^{t}\int_{S_{1}(t)}|\widehat{F}\cdot\bar{\widehat{u}}|+|\widehat{G}\cdot\bar{\widehat{\tau}}|d\xi dt^{\prime} \\\nonumber
\lesssim &\ |S_{1}(t)|^{\frac{1}{2}}\int_{0}^{t}\|F\|_{L^1}\|u\|_{L^2}+\|G\|_{L^1}\|\tau\|_{L^2}dt^{\prime}\\\nonumber
\lesssim &\ (1+t)^{-\frac{1}{2\beta}}\int_{0}^{t}\left(\|u\|_{L^2}^2+\|\tau\|_{L^2}^2\right)\left(\|\nabla u\|_{L^2}+\|\nabla\tau\|_{L^2}\right)dt^{\prime}\\\nonumber
\lesssim &\ (1+t)^{-\frac{1}{2\beta}+\frac{1}{2}}\left(\int_{0}^{t}D_{0}(s)dt^{\prime}\right)^{\frac{1}{2}}\\\nonumber
\lesssim &\ (1+t)^{-\frac{1}{2\beta}+\frac{1}{2}}.\nonumber
\end{align}
Inserting (\ref{326}) and (\ref{327}) into (\ref{325}), we have
\begin{align}\label{328}
{\int_{S_{1}(t)}|\hat{u}(\xi)|^{2}+|\hat \tau}(\xi)|^{2}d\xi\lesssim (1+t)^{-\frac{1}{2\beta}+\frac{1}{2}}.
\end{align}
Plugging (\ref{328}) into (\ref{321}) gives rise to 
\begin{align*}
  \frac{d}{dt}E_{0}(t)+\frac{C_{2}}{1+t}\left(\frac{k}{2}\|u\|_{H^{2\beta}}^{2}+\| \tau\|_{H^{2\beta}}^{2}\right)\lesssim (1+t)^{-(\frac{1}{2\beta}+\frac{1}{2})}.
  \end{align*}  
Consequently, we get the initial time decay rate
\begin{align}\label{3210}
E_{0}(t)\lesssim (1+t)^{-\frac{1}{2\beta}+\frac{1}{2}}.
  \end{align}
 
\textbf{Step 2: inital $\beta$-order decay rate}

Time decay rate in (\ref{3210}) is slow, in order to improve the time decay rate, we need to introduce $E_{\beta}$ and $D_{\beta}$:
\begin{align*}
E_{\beta}&=\left\|(\Lambda ^{\beta}u, \Lambda ^{\beta}\tau)\right\|_{H^{\beta}}^{2}+2 k\left\langle-\Lambda ^{2\beta-1}\nabla  u, \Lambda ^{2\beta-1}\tau\right\rangle_{H^{1-\beta}},\\
D_{\beta}&=\frac{k}{2}\left\|\nabla\Lambda^{2\beta-1} u\right\|_{H^{1-\beta}}^{2}+\left\|\Lambda^{2\beta} \tau\right\|_{H^{\beta}}^{2},
\end{align*}
where $k$ is a sufficiently small constant. Next we want to study the large-time behavior of $E_{\beta}$. Thanks to $\frac{1}{2}<\beta<1$, we obtain
\begin{align}\label{3211}
\quad&\frac{1}{2}\frac{d}{dt}\|(\Lambda^{\beta} u,\Lambda^{\beta} \tau)\|_{L^{2}}^{2}+\|\Lambda^{2\beta}\tau\|_{L^{2}}^{2}\\\nonumber
=&-\langle \Lambda^{\beta} (u\cdot \nabla u),\Lambda^{\beta} u \rangle
-\langle \Lambda^{\beta} (u\cdot \nabla \tau),\Lambda^{\beta} \tau\rangle
-\langle \Lambda^{\beta} Q(\nabla u,\tau), \Lambda^{\beta}\tau\rangle\\\nonumber
\lesssim &\ \|\Lambda^{2\beta} u\|_{L^{2}}\|u\|_{L^{\frac{1}{\beta-\frac{1}{2}}}}\|\nabla u\|_{L^{\frac{1}{1-\beta}}}
+\|\Lambda^{2\beta} \tau\|_{L^{2}}\|u\|_{L^{\frac{1}{\beta-\frac{1}{2}}}}\|\nabla \tau\|_{L^{\frac{1}{1-\beta}}}\\\nonumber
&+\|\Lambda^{2\beta} \tau\|_{L^{2}}\|\tau\|_{L^{\frac{1}{\beta-\frac{1}{2}}}}\|\nabla u\|_{L^{\frac{1}{1-\beta}}}\\\nonumber
\lesssim &\ \|u\|_{H^{2\beta}}\|\Lambda^{2\beta} u\|_{L^{2}}^{2}+\|u\|_{H^{2\beta}}\|\Lambda^{2\beta}\tau\|_{L^{2}}^{2}+\|\tau\|_{H^{2\beta}}\|\Lambda^{2\beta}\tau\|_{L^{2}}\|\Lambda^{2\beta}u\|_{L^{2}}\\
\lesssim &\ \delta D_{\beta}. \nonumber
\end{align}
Along the same line, for any $k>0$, we have
\begin{align}\label{3212}
&\frac{d}{dt}\langle\Lambda ^{2\beta-1} u,k\Lambda ^{2\beta-1}\mathrm{div}\tau\rangle+\frac{k}{2}\|\nabla\Lambda^{2\beta-1} u\|_{L^{2}}^{2} \\\nonumber
=\ &-k\langle\Lambda ^{2\beta-1}\mathbb{P}(u\cdot\nabla u),\Lambda ^{2\beta-1}\mathrm{div}\tau\rangle+k\langle\Lambda ^{2\beta-1}\mathbb{P}\mathrm{div}\tau,\Lambda ^{2\beta-1}\mathrm{div}\tau \rangle\\\nonumber
&-k\langle\Lambda ^{2\beta-1}\mathrm{div}(u\cdot\nabla\tau),\Lambda ^{2\beta-1}u\rangle
-k\langle \Lambda ^{2\beta-1}\mathrm{div}Q(\nabla u,\tau),\Lambda ^{2\beta-1}u\rangle\\\nonumber
&-k\langle\Lambda ^{2\beta-1}(-\Delta)^\beta\mathrm{div}\tau,\Lambda ^{2\beta-1}u\rangle\\\nonumber
\le\  & C\|u\|_{H^{s}}\|\Lambda^{2\beta}\tau\|_{L^{2}}\|\Lambda^{2\beta}u\|_{L^{2}}+C\|u\|_{H^{s}}\|\Lambda^{2\beta}u\|_{L^{2}}\|\Lambda^{2\beta}\tau\|_{L^{2}}\\\nonumber
&+k\|\tau\|_{H^{s}}\|\Lambda^{2\beta}u\|_{L^{2}}^2+k\|\Lambda^{2\beta}\tau\|_{L^{2}}^2+Ck\|\Lambda^{2\beta}u\|_{L^{2}}\|\Lambda^{4\beta-1}\tau\|_{L^{2}}\\\nonumber
\le\  & C\delta D_{\beta}+\frac{k}{100}\|\Lambda^{2\beta}u\|_{L^{2}}^{2}+kC\|\Lambda^{2\beta} \tau\|_{H^{\beta}}^{2}.\nonumber
\end{align}
By (\ref{212}) and (\ref{214}) in the Theorem \ref{1theo1}, we arrive at
\begin{align}\label{3213}
\quad\frac{1}{2}\frac{d}{dt}\|(\Lambda^{2\beta} u,\Lambda^{2\beta} \tau)\|_{L^{2}}^{2}+\|\Lambda^{3\beta}\tau\|_{L^{2}}^{2}
\lesssim \ \delta D_{\beta},
\end{align}
and
\begin{align}\label{3214}
\frac{d}{dt}\langle\Lambda ^{\beta} u,k\Lambda ^{\beta}\mathrm{div}\tau\rangle+\frac{k}{2}\|\nabla\Lambda^{\beta} u\|_{L^{2}}^{2} 
\le C\delta D_{\beta}+Ck\|\Lambda^{\beta+1}\tau\|_{H^{3\beta-1}}^{2}+\frac{k}{100}\|\Lambda^{\beta+1}u\|_{L^{2}}^{2}.
\end{align}
Combining (\ref{3211})-(\ref{3214}) and taking $k>0$ sufficiently small, then we find
\begin{align*}
\frac{d}{dt}E_{\beta}+D_{\beta}\le 0,
\end{align*}
which together with (\ref{3210}) leads to
\begin{align}\label{3216}
  \frac{d}{dt}E_{\beta}(t)+\frac{C_{2}}{1+t}\left(\frac{k}{2}\|\Lambda^{\beta}u\|_{H^{\beta}}^{2}+\|\Lambda^{\beta}\tau\|_{H^{\beta}}^{2}\right)\lesssim &(1+t)^{-1}\int_{S_{1}(t)}|\xi|^{2\beta}\left(|\widehat{u}(\xi)|^{2}+|\widehat {\tau}(\xi)|^{2}\right)d\xi\\ \nonumber
  \lesssim &(1+t)^{-\frac{1}{2\beta}-\frac{3}{2}}.\nonumber
  \end{align}  
We infer that
\begin{align}\label{3217}
  E_{\beta}(t)\lesssim (1+t)^{-\frac{1}{2\beta}-\frac{1}{2}}.
  \end{align} 

\textbf{Step 3: improve inital $0$-order time decay rate}

Plugging (\ref{324})-(\ref{326}) into (\ref{321}), we achieve
\begin{align*}
  \frac{d}{dt}E_{0}(t)&+\frac{C_{2}}{1+t}\left(\frac{k}{2}\|u\|_{H^{2\beta}}^{2}+\| \tau\|_{H^{2\beta}}^{2}\right)\lesssim (1+t)^{-1-\frac{1}{\beta}}\\\nonumber
  &+(1+t)^{-1-\frac{1}{2\beta}}\int_{0}^{t}\left(\|u\|_{L^2}^2+\|\tau\|_{L^2}^2\right)\left(\|\nabla u\|_{L^2}+\|\nabla\tau\|_{L^2}\right)dt^{\prime},\nonumber
  \end{align*} 
which implies
\begin{align}\label{3219}
(1+t)^{1+\frac{1}{2\beta}}\frac{d}{dt}E_{0}(t)&+C_{2}(1+t)^{\frac{1}{2\beta}}\left(\frac{k}{2}\|u\|_{H^{2\beta}}^{2}+\| \tau\|_{H^{2\beta}}^{2}\right)\\\nonumber
&\lesssim (1+t)^{-\frac{1}{2\beta}}+\int_{0}^{t}\left(\|u\|_{L^2}^2+\|\tau\|_{L^2}^2\right)\left(\|\nabla u\|_{L^2}+\|\nabla\tau\|_{L^2}\right)dt^{\prime}.\nonumber
\end{align}
Integrating (\ref{3219}) in time on $[0,t]$, we deduce that
\begin{align*}
(1+t)^{\frac{1}{2\beta}}E_{0}(t)\lesssim (1+t)^{-\frac{1}{2\beta}}+\int_{0}^{t}\left(\|u\|_{L^2}^2+\|\tau\|_{L^2}^2\right)\left(\|\nabla u\|_{L^2}+\|\nabla\tau\|_{L^2}\right)dt^{\prime}.
\end{align*}
Taking $\displaystyle N(t)=\sup_{0\le t^{\prime}\le t}(1+t^{\prime})^{\frac{1}{2\beta}}E_{0}(t^{\prime})$, then we find
\begin{align*}
N(t)\le C+C\int_{0}^{t}(1+t^{\prime})^{-\frac{1}{2\beta}}N(t^{\prime})(\|\nabla u\|_{L^2}+\|\nabla\tau\|_{L^2})dt^{\prime}.
\end{align*}
Applying Gronwall's inequality, we infer that for any $t>0$, $N(t)<\infty$, which gives rise to
%\begin{align}\label{364}
%N(t)\lesssim \exp\left(\displaystyle\int_{0}^{t}(1+t^{\prime})^{-\frac{3}{4\beta}-\frac{1}{4}}dt^{\prime}\right)<+\infty.
%\end{align}
\begin{align}\label{3222}
  E_{0}(t)\lesssim (1+t)^{-\frac{1}{2\beta}}.
\end{align}
Combining (\ref{3222}) with (\ref{3216}), we deduce that 
\begin{align}\label{3224}
  E_{\beta}(t)\lesssim (1+t)^{-\frac{1}{2\beta}-1}.
\end{align}

\textbf{Step 4: optimal $0$-order time decay rate  }

Then we prove the solution of (\ref{eq2}) belongs to some negative index Besov space. Applying $\dot\Delta_{j}$ to (\ref{eq2}), we find
\begin{align*}
 \left\{\begin{array}{l}\dot{\Delta}_{{j}}u_{t}+ \nabla\dot{\Delta}_{j}P-\text{div}\ \dot{\Delta}_{j}\tau=\dot{\Delta}_{j}F,\\ [1ex] \dot{\Delta}_{j}\tau_{t}+(-\Delta)^{\beta}\dot{\Delta}_{j}\tau-\dot{\Delta}_{j}D(u)=\dot{\Delta}_{j}G.\end{array}\right.
\end{align*}
Then we get
\begin{align}\label{3226}
\frac{d}{dt}(\|\dot{\Delta}_{j}u\|_{L^{2}}^{2}+\|\dot{\Delta}_{j}\tau\|_{L^{2} }^{2})+2\|\Lambda^{\beta}\dot{\Delta}_{j}\tau\|_{L^{2}}^{2}\lesssim \|\dot{\Delta}_{j}F\| _{L^{2}}\|\dot{\Delta}_{j}u\|_{L^{2}}+\|\dot{\Delta}_{j}G\|_{L^{2}}\|\dot{ \Delta}_{j}\tau\|_{L^{2}}.
\end{align}
Multiplying (\ref{3226}) by $2^{-2j}$ and taking $l^{\infty}$ norm, we have
\begin{align*}
\frac{d}{dt}\left(\|u\|_{\dot{B}^{-1}_{2,\infty}}^{2}+\|\tau\|_{\dot{B}^{-1}_{2,\infty}}^{2}\right)\lesssim \|F\|_{\dot{B}^{-1}_{2,\infty}}\|u\|_{ \dot{B}^{-1}_{2,\infty}}+\|G\|_{\dot{B}^{-1}_{2,\infty}}\|\tau\|_{ \dot{B}^{-1}_{2,\infty}}.  
\end{align*}
Let $\displaystyle M(t)=\sup_{0\le t^{\prime}\le t}\left(\|u\|_{\dot{B}^{-1}_{2,\infty}}+\|\tau\|_{\dot{B}^{-1}_{2,\infty}}\right)$, then we yield
\begin{align*}
M^{2}(t)\leq CM^{2}(0)+CM(t)\int_{0}^{t}\left(\|F\|_{\dot{B}^{-1}_{2,\infty}}+\|G\|_{\dot{B}^{-1}_{2,\infty}}\right)dt^{\prime}. 
\end{align*}
Note that $ L^{1}\hookrightarrow\dot{B}^{-1}_{2,\infty}$, we conclude, for any $t> 0$, that
\begin{align}\label{3229}
\int_{0}^{t}(\|F\|_{\dot{B}^{-1}_{2,\infty}}+\|G\|_{\dot{B}^{-1}_{2,\infty}})dt^{\prime}&\lesssim  \int_{0}^{t}(\|F\|_{L^1}+\|G\|_{L^1})dt^{\prime}\\\nonumber
&\lesssim  \int_{0}^{t}(\|u\|_{L^2}+\|\tau\|_{L^2})(\|\nabla u\|_{L^2}+\|\nabla\tau\|_{L^2})dt^{\prime}\\\nonumber
&\lesssim  \int_{0}^{t}(1+t^{\prime})^{-\frac{1}{2\beta}-\frac{1}{2}}dt^{\prime}<\infty.\nonumber
\end{align}
Hence, we deduce $M(t)<C$. From this we can obtain the optimal time decay rate for $E_{0}$. By (\ref{3222}), (\ref{3224}) and (\ref{3229}), we arrive at
\begin{align}\label{3230}
\int_{S_{1}(t)}\int_{0}^{t}|\widehat{F}\cdot\bar{\widehat{u}}|+|\widehat{G}\cdot\bar{\widehat{\tau}}|dt^{\prime}d\xi=& \int_{0}^{t}\int_{S_{1}(t)}|\widehat{F}\cdot\bar{\widehat{u}}|+|\widehat{G}\cdot\bar{\widehat{\tau}}|d\xi dt^{\prime} \\\nonumber
\lesssim & \int_{0}^{t}\left(\|F\|_{L^1}\int_{S_{1}(t)}|{\widehat{u}}|d\xi+\|G\|_{L^1}\int_{S_{1}(t)}|{\widehat{\tau}}|d\xi\right
) dt^{\prime} \\\nonumber
\lesssim & \ (1+t)^{-\frac{1}{2\beta}}\int_{0}^{t}\left[\left(\|F\|_{L^1}+\|G\|_{L^1}\right)\left(\int_{S_{1}(t)}|\widehat{u}|^{2}+|\widehat{\tau}|^{2}d\xi\right)^{\frac{1}{2}}\right]dt^{\prime}\\\nonumber
\lesssim &\ (1+t)^{-\frac{1}{\beta}}M(t)\int_{0}^{t}(1+t)^{-\frac{1}{2\beta}-\frac{1}{2}}dt^{\prime}\\\nonumber
\lesssim &\ (1+t)^{-\frac{1}{\beta}}.\nonumber
\end{align}
Thanks to (\ref{321}), (\ref{326}), (\ref{3216}) and (\ref{3230}), we achieve
\begin{align}\label{3231}
 E_{0}(t)\lesssim  (1+t)^{-\frac{1}{\beta}},\quad E_{\beta}(t)\lesssim  (1+t)^{-\frac{1}{\beta}-1},
  \end{align}
  which together with interpolation between Sobolev spaces gives the results.
 \end{proof}

 Now we introduce the energy and dissipation functionals for $(u,\tau)$ as follows:
 \begin{align*}
  \widetilde{E}_{s}&=(1+t)^{a^{\prime}}\left\|\Lambda^{s}(u, \tau)\right\|_{L^{2}}^{2}+k\left\langle\Lambda^{s-\beta} \tau,-\nabla \Lambda^{s-\beta} u\right\rangle,\\
\widetilde{D}_{s}&=(1+t)^{a^{\prime}}\|\Lambda^{s+\beta} \tau\|_{L^{2}}^{2}+\frac{k}{4}\|\nabla \Lambda^{s-\beta} u\|_{L^{2}}^{2},
 \end{align*}
 where $a^{\prime}=2-\frac{1}{\beta}\in [0,1)$ and $k>0$ is a small enough constant. Next we will prove the time decay rate for the highest derivative of the solution to (\ref{eq2}).
  \begin{prop}\label{3prop3}
Under the same conditions as in Proposition \ref{3prop2}, then there exists $C>0$ such that for every  $t>0$, there holds
\begin{align*}
\left\|\Lambda^{s}(u, \tau)\right\|_{L^{2}} \leq C(1+t)^{-\frac{s+1}{2 \beta}}.
\end{align*}  
\end{prop}
 
 \begin{proof}
 We also solely consider the case with $s=2\beta$. It is easy to derive 
 \begin{align}\label{3235}
 \frac{d}{dt}\widetilde{E}_{2\beta}+2\widetilde{D}_{2\beta}=&\ a^{\prime}(1+t)^{a^{\prime}-1}\|\Lambda^{2\beta}(u,\tau)\|_{L^{2}}^{2}+2(1+t)^ {a^{\prime}}(\langle\Lambda^{2\beta}F,\Lambda^{2\beta}u\rangle+\langle\Lambda^{s}G,\Lambda^{2\beta} \tau\rangle)\\ \nonumber
 &+k\langle(-\Lambda^{\beta}G-\Lambda^{3\beta}\tau),\nabla \Lambda^{\beta}u\rangle+k\langle\Lambda^{\beta}\mathbb{P}(F+\text{div }\tau),\text{div }\Lambda^{\beta}\tau\rangle.\nonumber
 \end{align}
Due to (\ref{3231}), we observe that
 \begin{align}\label{3236}
 (1+t)^{a}\langle\Lambda^{2\beta}F,\Lambda^{2\beta}u\rangle& \lesssim (1+t)^{a^{\prime}}\|\Lambda^{\beta+1} u\|_{L^{2}}^{2}\|\Lambda^{2\beta}u\|_{L^{2}}\\ \nonumber
 &\lesssim (1+t)^{a^{\prime}}(1+t)^{-\frac{1}{2}-\frac{1}{2\beta}}\|\Lambda^{ 1+\beta}u\|_{L^{2}}^{2}\\ \nonumber
 &\lesssim (1+t)^{\frac{3}{2}-\frac{3}{2\beta}}\|\Lambda^{\beta+1}u \|_{L^{2}}^{2}\\ \nonumber
 &\lesssim\delta \widetilde{D}_{2\beta}.\nonumber
 \end{align}
 According to Lemma \ref{lemma3}, we get
 \begin{align}\label{3238}
\|\Lambda^{1+\beta}\tau\|_{L^{2}}^{2}& \lesssim \|\tau\|_{L^{2}}^{\frac{2\beta-1}{3\beta}}\|\Lambda^{3\beta}\tau\|_{L^{2}}^{\frac{1+\beta}{3\beta}}\\\nonumber
 & \lesssim (1+t)^{-\frac{1}{\beta}-1}\|\tau\|_{L^{2}}^{2}+(1+t)^{a^{\prime}} \|\Lambda^{3\beta}\tau\|_{L^{2}}^{2}\\ \nonumber
 &\lesssim (1+t)^{-\frac{2}{\beta}-1}+(1+t)^{a^{\prime}}\|\Lambda^{3 \beta}\tau\|_{L^{2}}^{2},\nonumber
 \end{align} 
 from which we deduce that
 \begin{align}\label{3237}
(1+t)^{a^{\prime}}\langle\Lambda^{2\beta}G,\Lambda^{2\beta}\tau\rangle &\lesssim (1+t)^{a^{\prime}}\|\Lambda^{\beta}G\|_{L^{2}}\|\Lambda^{3\beta} \tau\|_{L^{2}}\\\nonumber
 &\lesssim (1+t)^{a^{\prime}}\|\Lambda^{3\beta}\tau\|_{L^{2}}(\|u\|_{L^{ \infty}}\|\Lambda^{\beta+1}\tau\|_{L^{2}}+\|\tau\|_{L^{\infty}}\|\nabla \Lambda^{\beta}u\|_{L^{2}})\\\nonumber
&\le\delta\widetilde{D}_{2\beta}+C(1+t)^{a^{\prime}}(\|u\|_{L^{ \infty}}^{2}\|\Lambda^{\beta+1}\tau\|_{L^{2}}^{2}+\|\tau\|_{L^{\infty}}^{2} \|\nabla\Lambda^{\beta}u\|_{L^{2}}^{2})\\ \nonumber
&\le 2\delta\widetilde{D}_{2\beta}+C(1+t)^{a^{\prime}-\frac{1}{\beta}}\| \Lambda^{\beta+1}\tau\|_{L^{2}}^{2}\\\nonumber
&\le 2\delta\widetilde{D}_{2\beta}+C(1+t)^{-\frac{2}{\beta}-1}.
 \end{align}
Similarly, we obtain
\begin{align*}
&k\left\langle  {\left( {-{\Lambda }^{\beta }G - {\Lambda }^{3\beta }\tau }\right) ,\nabla {\Lambda }^{\beta }u}\right\rangle   + k\left\langle  {{\Lambda }^{\beta }\mathbb{P}\left( {F + \operatorname{div}\tau }\right) ,\operatorname{div}{\Lambda }^{\beta }\tau }\right\rangle\\
\le &  C( \delta  + k) {\widetilde{D}}_{2\beta }+\frac{k}{100}\| \Lambda^{\beta+1}u\|_{L^{2}}^{2} + C (1 + t)^{-\frac{2}{\beta}-1},
\end{align*}
which together with (\ref{3236}) and (\ref{3237}) leads to 
\begin{align*}
\frac{d}{dt}{\widetilde{E}}_{2\beta } + 2{\widetilde{D}}_{2\beta } \leq  a^{\prime}{\left( 1 + t\right) }^{a^{\prime} - 1}\|{{\Lambda }^{2\beta}\left( u,\tau \right) }\|_{{L}^{2}}^{2} + C\left( {\delta  + k}\right) {\widetilde{D}}_{2\beta }+\frac{k}{100}\| \Lambda^{\beta+1}u\|_{L^{2}}^{2} + C{\left( 1 + t\right) }^{-\frac{2}{\beta}-1}.
\end{align*}
Recall that $S_{1}(t)=\left\{\xi:|\xi|^{2\beta}\le C_{2}(1+t)^{-1}\right\}$. For sufficiently large $C_2$, we conclude that
\begin{align*}
\frac{d}{dt}{\widetilde{E}}_{2\beta } + \frac{kC_{2}}{4}(1+t)^{a^{\prime}-1}\|\Lambda^{2\beta}(u,\tau)\|_{L^2} \lesssim  {\left( 1 + t\right) }^{-\frac{2}{\beta}-1}.
\end{align*}
Applying the time weighted energy estimate, taking  $kC_2$ big enough, by (\ref{3231}), then we obtain
\begin{align*}
  &(1+t)^{\frac{2}{\beta}+1}\widetilde{E}_{2\beta}+\int_{0}^{t}\frac{kC_{2}}{4}(1+t^{\prime})^{\frac{1}{\beta}+2}\|\Lambda^{2\beta}(u,\tau)\|_{L^{2}}^{2}dt^{\prime} \\
 & \lesssim (1+t)+\int_0^t(1+t^{\prime})^{\frac{2}{\beta}}\langle\Lambda^{\beta}\tau,-\nabla\Lambda^{\beta}u\rangle dt^{\prime}+\int_{0}^{t}(1+t^{\prime})^{\frac{1}{\beta}+2}\|\Lambda^{2\beta}(u,\tau)\|_{L^2}^2dt^{\prime} \\
 & \lesssim (1+t)+\int_0^t(1+t^{\prime})^{\frac{1}{\beta}+2}\|\Lambda^{2\beta}(u,\tau)\|_{L^2}^2dt^{\prime}+\int_0^t(1+t^{\prime})^{\frac{1}{\beta}-2}\|(u,\tau)\|_{L^2}^2dt^{\prime} \\
 & \lesssim (1+t)+\int_{0}^{t}(1+t^{\prime})^{\frac{1}{\beta}+2}\|\Lambda^{2\beta}(u,\tau)\|_{L^{2}}^{2}dt^{\prime}.
\end{align*}
By virtue of the time weighted energy estimate again, we achieve
\begin{align*}
(1+t)^{\frac{1}{\beta}+3}\|\Lambda^{2\beta}(u,\tau)\|_{L^{2}}^{2} 
&\lesssim (1+t)+k(1+t)^{\frac{2}{\beta}+1}\langle\Lambda^{\beta}\tau,-\nabla\Lambda^{\beta}u\rangle\\
 & \lesssim (1+t)+k(1+t)^{\frac{1}{\beta}+3}\|\Lambda^{2\beta}(u,\tau)\|_{L^2}^2+k(1+t)^{\frac{1}{\beta}-1}\|(u,\tau)\|_{L^2}^2 \\
 & \lesssim (1+t)+k(1+t)^{\frac{1}{\beta}+3}\|\Lambda^{2\beta}(u,\tau)\|_{L^{2}}^{2},
\end{align*}
which give rises to
\begin{align*}
\left\|\Lambda^{2\beta}(u, \tau)\right\|_{L^{2}} \leq C(1+t)^{-\frac{1}{2 \beta}-1}.
\end{align*}
This finishes the proof of Proposition \ref{3prop3}.
\end{proof}
\subsection{The lower bound}
Furthermore, we study the lower bound of the time decay rate. Our results can be stated as follows. 
\begin{prop}\label{3prop4}
  Under the same conditions as in Proposition \ref{3prop2}, if additionally $\left(u_{0}, \tau_{0}\right) \in H^{s+\beta}$  and  $0<\left|\int_{\mathbb{R}^{2}}\left(u_{0}, \tau_{0}\right) d x\right|$, then there exists  $C_{\beta}>0$ such that
\begin{align*}
\left\|\Lambda^{s}(u, \tau)\right\|_{L^{2}} \geq \frac{C_{\beta}}{2}(1+t)^{-\frac{s+1}{2 \beta}}.
\end{align*}
\end{prop}
\begin{proof}
We also consider the case for $s=2\beta$.
We first consider the linear system:
\begin{align}\label{341}
\left\{\begin{array}{l}\partial_{t}u_{L}+\nabla P_{L}-\mathrm{div}\ \tau_{L}=0,\ \mathrm{div}\ u_{L}=0,\\ [1ex] \partial_{t}\tau_{L}-D(u_{L})+(-\Delta)^{\beta}\tau_{L}=0,\\[1ex]
  u_{L}|_{t=0}=u_{0},\ \tau_{L}|_{t=0}=\tau_{0}.\end{array}\right.
\end{align}
According to Propositions \ref{3prop2}-\ref{3prop3}, one can arrive at $\|\Lambda^{s_1}(u_{L},\tau_{L})\|_{L^2}\le C(1+t)^{-\frac{1+s_1}{2\beta}}$ and $(u_{L},\tau_{L})\in L^{ \infty}\left([0,\infty),\dot{B}_{2,\infty}^{-1}\right)$. Applying Fourier transform to (\ref{341}), we get
\begin{align}\label{342}
\left\{\begin{array}{l}\partial_{t}\widehat{u}_{L}^{j}+i\xi_{j}\widehat{P}_{L}-i\xi _{k}\widehat{\tau}_{L}^{jk}=0,\\ [1ex] \partial_{t}\widehat{\tau}_{L}^{jk}+|\xi|^{2\beta}\widehat{\tau}_{L}^{jk}-\frac{i}{2}(\xi_ {k}\widehat{u}_{L}^{j}+\xi_{j}\widehat{u}_{L}^{k})=0.\end{array}\right.
\end{align}
Then we observe that
\begin{align*}
\frac{1}{2}\frac{d}{dt}\left[e^{2|\xi|^{2\beta}t}|(\widehat{u}_{L},\widehat{\tau}_{L})|^{2}\right]-| \xi|^{2\beta}e^{2|\xi|^{2\beta}t}|\widehat{u}_{L}|^{2}=0,
\end{align*}
which implies that
\begin{align*}
|\xi|^{2s_{1}}|(\widehat{u}_{L},\widehat{\tau}_{L})|^{2}=|\xi|^{2s_{1}}e^{-2|\xi|^{2\beta}t }|(\widehat{u}_{0},\widehat{\tau}_{0})|^{2}+\int_{0}^{t}2|\xi|^{2(s_{1}+\beta)}e^{-2|\xi|^{2\beta}(t-t^{\prime})}|\widehat{u}_{L}|^{2}dt^{\prime}.
\end{align*}
Due to the fact $0<c_{0}=|\int_{\mathbb{R}^{2}}(u_{0}, \tau_{0})dx|=|(\widehat{u}_{0}(0),\widehat{\tau}_{0}(0))|$, we deduce that there exists $\eta>0$ such that $|(\widehat{u}_{0}(\xi),\widehat{\tau}_{0}(\xi))|\geq\frac{c_{0}}{2}$ if $\xi\in B(0, \eta)$. Then we have 
 \begin{align}\label{343}
   \|(u_{L},\tau_{L})\|_{\dot{H}^{s_{1}}}^{2}& \geq\int_{|\xi|\leq\eta}|\xi|^{2s_{1}}e^{-2|\xi|^{2\beta}t}|(\widehat{u}_{0}, \widehat{\tau}_{0})|^{2}d\xi\\ \nonumber &\geq\frac{c_{0}^{2}}{4}\int_{|\xi|\leq\eta}|\xi|^{2s_{1}}e^{-2|\xi|^{2\beta}t}d\xi\\ \nonumber
   &\geq C_{0}^{2}(1+t)^{-\frac{1+s_{1}}{\beta}},\nonumber
 \end{align}
 where $C_0^2=\frac{c_{0}^{2}}{4}\int_{|y|\le \eta}|y|^{2s_{1}}e^{-2|y|^{2\beta}t}dy$.
 
 Taking $u_{N}=u-u_{L}, \tau_{N}= \tau-\tau_{L}$ and $P_{N}=P-P_{L}$, By Propositions \ref{3prop2}-\ref{3prop3}, we easily deduce $\|\Lambda^{s_{1}}(u_{N},\tau_{N})\|_{L^{2}}^{2} \leq C(1+t)^{-\frac{s_{1}+1}{\beta}}$ and $(u_{N},\tau_{N})\in L^{ \infty}\left([0,\infty),\dot{B}_{2,\infty}^{-1}\right)$,  Moreover, we have
 \begin{align}\label{344}
  \left\{\begin{array}{l}\partial_{t}u_{N}+\nabla P_{N}-\mathrm{div}\ \tau_{N}=F,\ \  \mathrm{div}\ u_{N}=0, \\ [1ex] \partial_{t}\tau_{N}-D(u_{N})+(-\Delta)^{\beta}\tau_{N}=G, \\[1ex]
   u_{N}|_{t=0}=\tau_{N}|_{t=0}=0.\end{array}\right. 
 \end{align}
According to the time decay rates for $(u_N,\tau_N)$ and $(u,\tau)$, we conclude from (\ref{344}) that
\begin{align*}
\frac{1}{2}\frac{d}{dt}\|(u_{N},\tau_{N})\|_{L^{2}}^ {2}+\|\Lambda^{\beta}\tau_{N}\|_{L^{2}}^{2}&=\langle F,u_{N}\rangle+ \langle G,\tau_{N}\rangle\\\nonumber
 &\lesssim \|\nabla u\|_{L^{2}}\|u\|_{L^{4}}\|u_{N}\|_{L^{4}}+\|\tau_ {N}\|_{L^{4}}\|\nabla(u,\tau)\|_{L^{2}}\|(u,\tau)\|_{L^{4}}\\ \nonumber
&\lesssim \delta(1+t)^{-\frac{1}{\beta}-1},\nonumber
\end{align*}
and
\begin{align*}
\frac{d}{dt}\langle\Lambda^{2\beta-2}\tau_{N},-\nabla u_{N}\rangle+ \frac{1}{2}\|\nabla\Lambda^{\beta-1} u_{N}\|_{L^{2}}^{2}=&\langle\Lambda^{2\beta-2}( G-(-\Delta)^{\beta}\tau_{N}),-\nabla u_{N}\rangle\\\nonumber
&-\langle\Lambda^{2\beta-2}\ \mathbb{P}(F+\text{div }\tau_{N}),\text{div }\tau_{N}\rangle\\\nonumber
\lesssim& (1+t)^{-\frac{1}{\beta}-1},\nonumber
\end{align*}
which lead to
\begin{align*}
\frac{d}{dt}[\|(u_{N},\tau_{N})\|_{L^{2}}^ {2}+2k_{0}\langle\Lambda^{2\beta-2}\tau_{N},-\nabla u_{N}\rangle]+2\|\Lambda^{\beta}\tau_{N}\|_{L^{2}}^{2}+ k_{0}\|\nabla^{\beta} u_{N}\|_{L^{2}}^{2}&\le\ C(\delta+k_{0})(1+t)^{-1-\frac{1}{\beta}}.
\end{align*}
Then we conclude that
\begin{align}\label{345}
&\frac{d}{dt}\left[\|(u_{N},\tau_{N})\|_{L^{2}}^{2}+2k_{0}\langle\Lambda^{2\beta-2}\tau_{N},-\nabla u_{N}\rangle\right]+\frac{ k_{0}C_{2}}{(1+t)}\|u_{N}\|_{L^{2}}^{2}+\frac{2C_{2}}{1+t}\|\tau_{N}\|_{L^{2}}^{2} \\\nonumber
 \lesssim &\frac{C_{2}}{1+t}\int_{S_{1}(t)}|\widehat{u_{N}}(\xi)|^{2}+|\widehat{ \tau_{N}}(\xi)|^{2}d\xi+(\delta+k_{0})(1+t)^{-\frac{1}{\beta}-1}.\nonumber
\end{align}
We deduce from a similar derivation of (\ref{328}) that
\begin{align}\label{346}
\int_{S_{1}(t)}|\widehat{u_{N}}|^{2}+|\widehat{\tau_{N}}|^{2}d \xi&\lesssim \int_{S_{1}(t)}\int_{0}^{t}|\widehat{F}\cdot\overline{\widehat{u_ {N}}}|+|\widehat{G}\cdot\overline{\widehat{\tau_{N}}}|dt^{\prime}d\xi\\ \nonumber
&\lesssim (1+t)^{-\frac{1}{\beta}}\int_{0}^{t}\|(u,\tau)\|_{L^{2}}\|\nabla(u,\tau) \|_{L^{2}}\|(u_{N},\tau_{N})\|_{\dot{B}^{-1}_{2,\infty}}dt^{\prime}\\ &\lesssim \delta(1+t)^{-\frac{1}{\beta}}.\nonumber
\end{align}
Plugging (\ref{346}) into (\ref{345}), we find
\begin{align*}
&\frac{d}{dt}[\|(u_{N},\tau_{N})\|_{L^{2}}^{2}+2k_{0}\langle\Lambda^{2\beta-2}\tau_{N},-\nabla u_{N}\rangle]+\frac{ k_{0}C_{2}}{2(1+t)}\|u_{N}\|_{L^{2}}^{2}+\frac{C_{2}}{1+t}\|\tau_{N}\|_{L^{2}}^{2} \\ \nonumber
\lesssim& \left(\delta C_2+k_{0}\right)(1+t)^{-\frac{1}{\beta}-1}.\nonumber
\end{align*}
Thanks to Propositions \ref{3prop2}-\ref{3prop3} again, we arrive at
\begin{align*}
&(1+t)^{1+\frac{1}{\beta}}\|(u_{N},\tau_{N})\|_{L^{2}}^{2}\\\nonumber
\lesssim &\ \int_{0}^{t}2k_{0}(1+t^{\prime})^{\frac{1}{\beta}}\langle\Lambda^{2\beta-2}\tau_{N},-\nabla u_{N}\rangle dt^{\prime}
+2k_{0}(1+t)^{1+\frac{1}{\beta}}\langle\Lambda^{2\beta-2}\tau_{N},\nabla u_{N}\rangle+(\delta C_2+k_{0})(1+t)\\\nonumber
 \lesssim &\ \left(\delta C_2+k_{0}\right)(1+t),
\end{align*}
which yields
\begin{align}\label{347}
\|(u_{N},\tau_{N})\|_{L^{2}}^{2}\le C(\delta C_2+k_{0})(1+t)^{-\frac{1}{\beta}}.
\end{align}
Applying $\Lambda^{s_{1}}$ to (\ref{344}), we get
\begin{align*}
\left\{\begin{array}{l}\partial_{t}\Lambda^{s_{1}}u_{N}+\nabla\Lambda^{s_{1 }}P_{N}- \mathrm{div} \Lambda^{s_{1}}\tau_{N}=\Lambda^{s_{1}}F,\\ [1ex] \partial_{t}\Lambda^{s_{1}}\tau_{N}-D(\Lambda^{s_{1}}u_{N})+(-\Delta)^{\beta}\Lambda^{s_{1}}\tau_{N}=\Lambda^{s_{1}}G.\end{array}\right.
\end{align*}
Standard energy estimate yields 
 \begin{align}\label{349}
&\frac{1}{2} \frac{d}{d t}\left[(1+t)^{a^{\prime}}\left\|\Lambda^{2\beta}(u_{N}, \tau_{N})\right\|_{L^{2}}^{2}\right]+(1+t)^{a^{\prime}}\left\|\Lambda^{3\beta} \tau_{N}\right\|_{L^{2}}^{2}\\\nonumber
&=-a^{\prime}(1+t)^{a^{\prime}-1}\left\|\Lambda^{2\beta}(u_{N}, \tau_{N})\right\|_{L^{2}}^{2}+(1+t)^{a^{\prime}}\left\langle\Lambda^{2\beta} F, 
\Lambda^{2\beta} u_{N}\right\rangle+(1+t)^{a^{\prime}}\left\langle\Lambda^{2\beta} G, \Lambda^{2\beta} \tau_{N}\right\rangle.\nonumber
 \end{align}
Notice that
 \begin{align*}
 \left\langle\Lambda^{2\beta} F, \Lambda^{2\beta} u_{N}\right\rangle=-\left\langle\Lambda^{2\beta}(u\cdot\nabla u), \Lambda^{2\beta} u_{N}\right\rangle=-\left\langle\Lambda^{2\beta}(u\cdot\nabla u), \Lambda^{2\beta} u\right\rangle+\left\langle\Lambda^{2\beta}(u\cdot\nabla u), \Lambda^{2\beta} u_{L}\right\rangle.
 \end{align*}
  Using Lemma \ref{lemma4}, Lemma \ref{lemma5} and Proposition \ref{3prop3}, we infer that
  \begin{align*}
(1+t)^{a^{\prime}}\left\langle\Lambda^{2\beta}(u\cdot\nabla u), \Lambda^{2\beta} u\right\rangle & \lesssim (1+t)^{a^{\prime}}\|\Lambda^{\beta+1}u\|_{L^2}^{2}\|\Lambda^{2\beta}u\|_{L^2}\\\nonumber
& \lesssim (1+t)^{a^{\prime}}\|\Lambda^{2\beta}u\|_{L^2}\left(\|\Lambda^{\beta+1}u_{N}\|_{L^2}^{2}+\|\Lambda^{\beta+1}u_{L}\|_{L^2}^{2}\right)\\\nonumber
& \lesssim \delta\|\Lambda^{\beta+1}u_{N}\|_{L^2}^{2}+\delta(1+t)^{-\frac{2}{\beta}-1},\nonumber
  \end{align*}
  and
  \begin{align*}
(1+t)^{a^{\prime}}\left\langle\Lambda^{2\beta}(u\cdot\nabla u), \Lambda^{2\beta} u_{L}\right\rangle&=(1+t)^{a^{\prime}}\left\langle\Lambda^{\beta}(u\cdot\nabla u), \Lambda^{3\beta} u_{L}\right\rangle\\\nonumber
& \lesssim (1+t)^{a^{\prime}}\|\Lambda^{\beta+1}u\|_{L^2}\|u\|_{L^\infty}\|\Lambda^{3\beta}u_{L}\|_{L^2}\\\nonumber
& \lesssim (1+t)^{a^{\prime}}(\|\Lambda^{\beta+1}u_{L}\|_{L^2}+\|\Lambda^{\beta+1}u_{N}\|_{L^2})\|u\|_{L^\infty}\|\Lambda^{3\beta}u_{L}\|_{L^2}\\\nonumber
& \lesssim \delta(1+t)^{-\frac{2}{\beta}-1}+\delta\|\Lambda^{\beta+1}u_{N}\|_{L^2}^{2}.\nonumber
  \end{align*}
By summarizing the above estimates, we conclude that
  \begin{align}\label{3410}
   (1+t)^{a^{\prime}}\left\langle\Lambda^{2\beta} F, \Lambda^{2\beta} u_{N}\right\rangle\lesssim \delta(1+t)^{-\frac{2}{\beta}-1}+\delta\|\Lambda^{\beta+1}u_{N}\|_{L^2}^{2}.
  \end{align}
By (\ref{3238}) and time decay rate for $\tau_N$, there holds
\begin{align*}
(1+t)^{a^{\prime}}\|u\|_{L^{\infty}}\|\Lambda^{\beta+1}\tau\|_{L^{2}}^{2}
& \lesssim (1+t)^{a^{\prime}}\|u\|_{L^{\infty}}\|\tau\|_{L^{2}}^{\frac{2(2\beta-1)}{3\beta}}\|\Lambda^{3\beta}\tau\|_{L^{2}}^{\frac{2(\beta+1)}{3\beta}}\\
& \lesssim \delta(1+t)^{a^{\prime}}\|\Lambda^{3\beta}\tau\|_{L^{2}}^{2}+\delta(1+t)^{-\frac{2}{\beta}-1}\\ \nonumber
& \lesssim \delta(1+t)^{a^{\prime}}\|\Lambda^{3\beta}\tau_N\|_{L^{2}}^{2}+\delta(1+t)^{-\frac{2}{\beta}-1}.
\end{align*}  
Then, we have 
\begin{align}\label{3411}
(1+t)^{a^{\prime}}\left\langle\Lambda^{2\beta} G, \Lambda^{2\beta} \tau_{N}\right\rangle &=(1+t)^{a^{\prime}}\left\langle\Lambda^{\beta}G, \Lambda^{3\beta} \tau_{N}\right\rangle\\\nonumber
& \lesssim  (1+t)^{a^{\prime}}\|\Lambda^{3\beta}\tau_{N}\|_{L^{2}}(\|\Lambda^{\beta+1}u\|_{L^{2}}\|\tau\|_{L^{\infty}}+\|\Lambda^{\beta+1}\tau\|_{L^{2}}\|u\|_{L^{\infty}})\\\nonumber
& \lesssim \delta(1+t)^{a^{\prime}}\|\Lambda^{3\beta}\tau_{N}\|_{L^{2}}^{2}+\delta(1+t)^{-\frac{2}{\beta}-1}+\delta\|\Lambda^{\beta+1}u_{N}\|_{L^{2}}^{2}.\nonumber
  \end{align}
Plugging  (\ref{3410}) and (\ref{3411}) into (\ref{349}) gives rise to
 \begin{align}\label{3412}
&\frac{1}{2} \frac{d}{d t}\left[(1+t)^{a^{\prime}}\left\|\Lambda^{2\beta}(u_{N}, \tau_{N})\right\|_{L^{2}}^{2}\right]+(1+t)^{a^{\prime}}\left\|\Lambda^{3\beta} \tau_{N}\right\|_{L^{2}}^{2}\\\nonumber
\le &\ a^{\prime}(1+t)^{a^{\prime}-1}\left\|\Lambda^{2\beta}(u_{N}, \tau_{N})\right\|_{L^{2}}^{2}+C\delta\|\Lambda^{\beta+1}u_{N}\|_{L^2}^{2}\\\nonumber
&+C\delta(1+t)^{a^{\prime}}\|\Lambda^{3\beta}\tau_N\|_{L^{2}}^{2}+C\delta(1+t)^{-\frac{2}{\beta}-1}.\nonumber
 \end{align}
Along the same line to the proof of (\ref{3412}), we achieve
\begin{align}\label{3413}
  &\frac{d}{dt}\langle\Lambda^{\beta}\tau_{N},-\nabla \Lambda^{\beta}u_{N}\rangle+\frac{1}{2}\|\nabla\Lambda^{\beta}u_{N}\|_{L^{2}}^{2} \\ \nonumber =&\langle\Lambda^{\beta}(G-(-\Delta)^{\beta}\tau_{N}),-\nabla\Lambda^{\beta}u_{N }\rangle+\langle\Lambda^{\beta}\mathbb{P}(F+\mathrm{div}\ \tau_{N}),\mathrm{div}\ \Lambda^{\beta}\tau_{N}\rangle\\ \nonumber
  \leq&\left(\frac{1}{100}+C\delta\right)\|\nabla\Lambda^{\beta}u_{N}\|_{L^{2}}^{2}+C\| \Lambda^{3\beta}\tau_{N}\|_{L^{2}}^{2}+C\| \Lambda^{\beta+1}\tau_{N}\|_{L^{2}}^{2}+C\delta(1+t)^{-\frac{2}{\beta}-1}.\nonumber
\end{align}
Combining (\ref{3412}) and (\ref{3413}), we get
\begin{align}\label{3414}
&\frac{d}{dt}\left[(1+t)^{a^{\prime}}\|\Lambda^{2\beta}(u_{N},\tau_{N})\|_{L^{2}}^{2}+4 \langle\Lambda^{\beta}\tau_{N},-\nabla\Lambda^{\beta}u_{N}\rangle\right]\\\nonumber
&+\frac{3}{4}(1+t)^{a^{\prime}}\| \Lambda^{3\beta}\tau_{N}\|_{L^{2}}^{2}+\|\nabla\Lambda^{\beta}u_{N}\|_{L^{2}}^{2}\\\nonumber
 \le& a^{\prime}(1+t)^{a^{\prime}-1}\|\Lambda^{2\beta}(u_{N},\tau_{N})\|_{L^{2}}^{2}+C\| \Lambda^{3\beta}\tau_{N}\|_{L^{2}}^{2}+C\delta(1+t)^{-\frac{2}{\beta}-1},\nonumber
\end{align} 
where we have used the fact that
\begin{align*}
\|\Lambda^{\beta+1}\tau_N\|_{L^{2}}^{2}
\lesssim&  \|\tau_N\|_{L^{2}}^{\frac{2(2\beta-1)}{3\beta}}\|\Lambda^{3\beta}\tau_N\|_{L^{2}}^{\frac{2(1+\beta)}{3\beta}}\\\nonumber
\le&  \frac{1}{4C}(1+t)^{a^{\prime}}\|\Lambda^{3\beta}\tau_N\|_{L^{2}}^{2}+C(\delta C_2+k_0)(1+t)^{-\frac{2}{\beta}-1}.\nonumber
\end{align*}
It follows from (\ref{3414}) that
\begin{align*}
&\frac{d}{dt}\left[(1+t)^{a^{\prime}}\|\Lambda^{2\beta}(u_{N},\tau_{N})\|_{L^{2}}^{2}+4 \langle\Lambda^{\beta}\tau_{N},-\nabla\Lambda^{s-\beta}u_{N}\rangle\right]+\frac{C_{2}}{2}(1+t)^{a^{\prime}-1}\|\Lambda^{2\beta}(\tau_{N},u_{N})\|_{L^2}^{2}\\\nonumber
\le&  C_{2}(1+t)^{a^{\prime}-1}\int_{S_{\beta}(t)}|\xi|^{4\beta}\left(|\widehat{\tau_{N}}(\xi)|^{2}+|\widehat{u_{N}}(\xi)|^{2}\right)d \xi+C[\delta+(\delta C_2+k_0)](1+t)^{-\frac{2}{\beta}-1}\\\nonumber
 \lesssim& \left[\delta+(\delta C_{2}+k_{0})C_{2}\right](1+t)^{-\frac{2}{\beta}-1}.
\end{align*}
Thanks to (\ref{347}),  we infer
\begin{align*}
&(1+t)^{\frac{2}{\beta}+1}\left[(1+t)^{a^{\prime}}\|\Lambda^{2\beta}(u_{N},\tau_{N})\|_{L^ {2}}^{2}+4\langle\Lambda^{\beta}\tau_{N},-\nabla\Lambda^{\beta}u_{N}\rangle\right]\\\nonumber
&+\frac{1}{2}\int_{0}^{t}  C_{2}\left(1+t^{\prime}\right)^{\frac{1}{\beta}+2}\left\|\Lambda^{2\beta}(\tau, u)\right\|_{L^{2}}^{2} d t^{\prime}\\\nonumber
\lesssim &\int_{0}^{t}(1+t^{\prime})^{{\frac{2}{\beta}}}\langle\Lambda^{\beta}\tau_{N},\nabla \Lambda^{\beta}u_{N}\rangle+(1+t^{\prime})^{{\frac{1}{\beta}+2}}\|\Lambda^{2\beta}(u_{N},\tau_{N})\|_{L^ {2}}^{2} d t^{\prime}+[\delta+(\delta C_{2}+k_{0})C_{2}](1+t)\\ \nonumber
\lesssim &\int_{0}^{t}(1+t^{\prime})^{\frac{1}{\beta}+2}\|\Lambda^{2\beta}(u_N,\tau_N)\|_{L^2}^2+(1+t^{\prime})^{\frac{1}{\beta}-2}\|(u_N,\tau_N)\|_{L^2}^2 d t^{\prime}+[\delta+(\delta C_{2}+k_{0})C_{2}](1+t)\\\nonumber
\lesssim &\int_{0}^{t}(1+t^{\prime})^{\frac{1}{\beta}+2}\|\Lambda^{2\beta}(u_N,\tau_N)\|_{L^2}^2+C(\delta C_2+k_{0})(1+t^{\prime})^{-2} d t^{\prime}+[\delta+(\delta C_{2}+k_{0})C_{2}](1+t)\\\nonumber
\lesssim &\int_{0}^{t}(1+t^{\prime})^{\frac{2}{\beta}+1}\|\Lambda^{2\beta}(u_N,\tau_N)\|_{L^2}^2 d t^{\prime}+[\delta+(\delta C_{2}+k_{0})C_{2}](1+t).\nonumber
\end{align*}
Taking $C_{2}$ big enough, we infer that
\begin{align*}
&\quad(1+t)^{\frac{1}{\beta}+3}\|\Lambda^{2\beta}(u_{N},\tau_{N})\|_{L^{2}}^{2}\\\nonumber
& \leq 4(1+t)^{\frac{2}{\beta}+1}\langle\Lambda^{\beta}\tau_{N},\nabla \Lambda^{\beta}u_{N}\rangle+C[\delta+(\delta C_{2}+k_{0})C_{2}](1+t)\\ \nonumber
& \leq \frac{1}{2} (1+t)^{\frac{1}{\beta}+3}\|\Lambda^{2\beta}(u_N,\tau_N)\|_{L^2}^2+C(1+t)^{\frac{1}{\beta}-1}\|(u_N,\tau_N)\|_{L^2}^2 +C[\delta+(\delta C_{2}+k_{0})C_{2}](1+t)\\\nonumber
&\leq \frac{1}{2} (1+t)^{\frac{1}{\beta}+3}\|\Lambda^{2\beta}(u_N,\tau_N)\|_{L^2}^2+C[\delta+(\delta C_{2}+k_{0})C_{2}] (1+t)^{-1}+C[\delta+(\delta C_{2}+k_{0})C_{2}](1+t)\\\nonumber
&\leq \frac{1}{2} (1+t)^{\frac{1}{\beta}+3}\|\Lambda^{2\beta}(u_N,\tau_N)\|_{L^2}^2+C[\delta+(\delta C_{2}+k_{0})C_{2}](1+t),\nonumber
\end{align*}
which gives rise to 
\begin{align}\label{3415}
\|\Lambda^{2\beta}(u_{N},\tau_{N})\|_{L^{2}}^{2}\le C[\delta+(\delta C_{2}+k_{0})C_{2}](1+t)^{-\frac{1}{\beta}-2}.
\end{align}
Taking $\delta+(\delta C_{2}+k_{0})C_{2}$ small enough, we arrive at
\begin{align}\label{3416}
\|\Lambda^{2\beta}(u_{N},\tau_{N})\|_{L^{2}}^{2}\le \frac{C_{0}^{2}}{4}(1+t)^{-\frac{1}{\beta}-2}.
\end{align}
Combining (\ref{343}) and (\ref{3416}), it comes out
\begin{align*}
\|\Lambda^{2\beta}(u,\tau)\|_{L^{2}}^{2}\ge \frac{C_{0}^{2}}{4}(1+t)^{-\frac{1}{\beta}-2}.
\end{align*}
\end{proof}
\textbf{Proof of Theorem \ref{1theo2}:} \\
By summarizing the Propositions \ref{3prop2}-\ref{3prop4}, we complete the proof of Theorem \ref{1theo2} \hfill$\Box$

	\section{Uniform vanishing damping limit}
\par
In this section, we consider the uniform vanishing damping limit for (\ref{eq2}). In Section 5.1, we will present the proof of Theorem \ref{1theo3}. In Section 5.2, we establish the uniform vanishing damping limit for (\ref{eq2}), namely Theorem \ref{1theo5}. In Section 5.3, we study the time decay rate of $\mathrm{tr}\tau^a$. Moreover, we can establish sharp uniform vanishing damping rate for $\mathrm{tr}\tau^a$, which is Theorem \ref{1theo6}.

\subsection{Time decay rates with $a\in (0,1]$ and the time integrability}
\par
We firstly study the large-time behavior of (\ref{eq2}) with $a\in (0,1]$. Due to the  presence of damping term, the estimate of energy $E_{0}$ defined in Section 3.1 is not closed. However, we can establish a time decay rate for (\ref{eq2}) by virtue of mix-order Fourier splitting method. Note that the time decay rate is uniform with respect to $a\in (0,1]$. 
Denote that
\begin{align*}
\bar{E}_{\theta}&=\|\Lambda^{\theta}(u^a,\tau^a)\|_{H^{s-\theta}}^{2}+2k \langle\nabla\Lambda^{\theta}u^a,\Lambda^{\theta}\tau^a\rangle_{H^{s-\beta-\theta}},\newline \\ \bar{D}_{\theta}&=\frac{k}{2}\|\nabla\Lambda^{\theta}u^a\|_{H^{s-\beta-\theta }}^{2}+\|\Lambda^{\beta+\theta}\tau^a\|_{H^{s-\theta}}^{2},
\end{align*}
where $\theta=0,1$. We have the following proposition.
\begin{prop}\label{4prop1}
 Let $d=2,\ a\in [0,1]$, $\frac{1}{2}\le\beta<1$ and $s\ge 1+\beta$. Let $(u^a,\tau^a)$  be a strong solution of (\ref{eq2}) with the initial data $(u_{0},\tau_{0})\in H^s$ under the condition in Theorem \ref{1theo1}. If  $\left(u_{0}, \tau_{0}\right) \in \dot{B}_{2, \infty}^{-1}$, then there exists a positive constant $C$  such that
 \begin{align*}
\left\|\Lambda^{s_1}(u^a, \tau^a)\right\|_{ H^{s-s_1}} \leq C(1+t)^{-\frac{1+s_1}{2}},
 \end{align*}
 where $s_{1}\in [0,1]$.
\end{prop} 
\begin{proof}
We only prove the case with $1+\beta\le s\le2$, The proof for case with $s>2$ can be conducted by the method in \cite{LLY}. We now divide the proof of Proposition \ref{4prop1} into the following four steps.

\textbf{Step1: inital 0-order time decay rate}

Taking $\theta=0$, we deduce from Theorem \ref{1theo1} that
\begin{align}\label{401}
\frac{d}{dt}\bar{E}_{0}+\bar{D}_{0}\le 0.
\end{align}
Denote that $S_{0}(t)=\left\{\xi:|\xi|^{2}\leq 2C_{2}\frac{f^{ \prime}(t)}{f(t)}\right\}$, where $f(t)=\ln^{3}(e+t)$ and $C_{2}$ is large enough. By (\ref{401}),
we deduce that
\begin{align}\label{402}
\frac{d}{dt}[f(t)\bar{E}_{0}(t)]+C_{2}f^{\prime}(t)\left(\frac{k}{2}\|u^a\|_{H^{s}}^{2}+\| \tau^a\|_{H^{s}}^{2}\right)\lesssim f^{\prime}(t)\int_{S_{0}(t)}|\widehat{u^a}(\xi)|^{2}+|\widehat{\tau^a}(\xi)|^{2}d\xi.
\end{align}
Applying
Fourier transform to (\ref{eq2}), we have
\begin{align}\label{403}
\left\{\begin{array}{l}\widehat{u_{t}}^{j}+i\xi_{j}\widehat{P}-i\xi_{k}\widehat{\tau}^{ jk}=\widehat{F}^{j},\\ [1ex] \widehat{\tau}^{jk}_{t}+|\xi|^{2\beta}\widehat{\tau}^{jk}+a\tau^{jk}-\frac{i}{2}(\xi_{k}\widehat{u}^{j}+ \xi_{j}\widehat{u}^{k})=\widehat{G}^{jk},\end{array}\right.
\end{align}
Integrating (\ref{403}) in time on $[0,t]$, we get
\begin{align}\label{404}
|\widehat{u^a}|^{2}+|\widehat{\tau^a}|^{2}\lesssim (|\widehat{u}_{0}|^{2}+|\widehat{\tau}_{0}|^{2})+\int_{0}^{t}|\widehat{F}\cdot\overline{\widehat{u^a}}|+|\widehat{G}\cdot\overline{\widehat{\tau^a}}|dt^{\prime}.
\end{align}
Integrating (\ref{404}) over $S_{0}(t)$ with $\xi$, then we derive
\begin{align}\label{405}
\int_{S_{0}(t)}|\widehat{u^a}|^{2}+|\widehat{\tau^a}|^{2}d\xi\lesssim \int_{S_{0}(t)}|\widehat{u}_{0}|^{2}+|\widehat{\tau}_{0}|^{2}d\xi+\int_{S_{0}(t)}\int_{0}^{t}|\widehat{F}\cdot\overline{\widehat{u^a}}|+|\widehat{G}\cdot\overline{\widehat{\tau^a}}|dt^{\prime}d\xi.   
\end{align}
Due to the fact $E(0)<\infty$ and $(u_0,\tau_0)\in \dot B_{2,\infty}^{-1}$, we find
\begin{align}\label{406}
\int_{S_{0}(t)}(|\hat{u}_{0}|^{2}+|\hat{\tau}_{0}|^{2})d\xi &\lesssim\sum_{j\leq\log_{2}\left[\frac{4}{3}C_{2}^{\frac{1}{2}}\sqrt{\frac{f^{\prime}(t)}{f(t)}}\right]}\left(\|\dot{\Delta}_{j}u_{0}\|_{L^{2}}^{2}+\|\dot{\Delta}_{j}\tau_{0}\|_{L^{2}}^{2}\right) \\ \nonumber
 & \lesssim\ \frac{f^{\prime}(t)}{f(t)}\|(u_{0},\tau_{0})\|_{\dot{B}_{2,\infty}^{-1}}^{2}.\nonumber
\end{align}
By Minkowski's inequality, we arrive at
\begin{align}\label{407}
\int_{S_{0}(t)}\int_{0}^{t}|\widehat{F}\cdot\bar{\widehat{u}}|+|\widehat{G}\cdot\bar{\widehat{\tau}}|dt^{\prime}d\xi=& \int_{0}^{t}\int_{S_{0}(t)}|\widehat{F}\cdot\bar{\widehat{u}}|+|\widehat{G}\cdot\bar{\widehat{\tau}}|d\xi dt^{\prime} \\\nonumber
\lesssim &\ \frac{f^{\prime}(t)}{f(t)}\int_{0}^{t}\left(\|u\|_{L^2}^2+\|\tau\|_{L^2}^2\right)\left(\|\nabla u\|_{L^2}+\|\nabla\tau\|_{L^2}\right)dt^{\prime}\\\nonumber
\lesssim &\ \frac{f^{\prime}(t)}{f(t)}(1+t)^{\frac{1}{2}}.\nonumber
\end{align}
By inserting (\ref{406}) and (\ref{408}) into (\ref{405}), we obtain
\begin{align}\label{408}
\int_{S_{0}(t)}|\widehat{u^a}|^{2}+|\widehat{\tau^a}|^{2}d\xi\lesssim \ln^{-\frac{1}{2}}(e+t).
\end{align}
Plugging (\ref{408}) into (\ref{402}), we have
\begin{align*}
\frac{d}{dt}[f(t)\bar{E}_{0}(t)]\lesssim f^{\prime}(t)\ln^{-\frac{1}{2}}(e+t).
\end{align*}
Consequently, we get the initial time decay rate
\begin{align*}
\bar{E}_{0}(t)\lesssim \ln^{-\frac{1}{2}}(e+t).
\end{align*}
Using the bootstrap argument from \cite{LUO2025113443}, for any $l\in N^{+}$, we obtain
\begin{align}\label{410}
\bar{E}_{0}(t)\lesssim \ln^{-l}(e+t).
\end{align}

\textbf{Step2: inital $1$-order decay rate}

Similar to the proof of (\ref{3211}), one has
\begin{align}\label{411}
\frac{1}{2}\frac{d}{dt}\|\nabla(u^a,\tau^a)\|_{L^{2}}^{2}+\|\Lambda^{\beta+1}\tau^a\|_{L^{ 2}}^{2}+a\|\nabla \tau^a\|_{L^2}^{2}&=-\langle\nabla\left( u^a\cdot\nabla\tau+Q(\nabla u^a,\tau^a)\right),\nabla\tau^a\rangle\\ \notag
&\lesssim  \delta \bar{D}_{1}.
\end{align}
Thanks to $s\ge1+\beta$, by applying Lemma \ref{lemma5}, for any $k>0$, we obtain
\begin{align}\label{412}
&\frac{d}{dt}\langle\Delta\tau^a,k\nabla u^a\rangle+ \frac{k}{2}\|\nabla^{2}u^a\|_{L^{2}}^{2}\\\nonumber
=&\ \langle\text{div }(u^a\cdot\nabla\tau^a+Q(\nabla u^a,\tau^a)+a\tau^a+(-\Delta)^{\beta}\tau^a),k\Delta u^a\rangle
-\langle\nabla\mathbb{P}(u^a\cdot\nabla u^a-\text{div }\tau^a),k\Delta\tau^a\rangle\\ \nonumber
\le &\ k\|\tau^a\|_{L^{\infty}}\|\nabla^{2} u^a\|_{L^{2}}\|\Delta u^a\|_{L^{2}}+k\|u^a\|_{L^{\infty}}\|\nabla^{2} \tau^a\|_{L^{2}}\|\Delta u^a\|_{L^{2}}
+k\|\nabla u^a\|_{L^{\frac{2}{\beta}}}\|\nabla \tau^a\|_{L^{\frac{2}{1-\beta}}}\|\Delta u^a\|_{L^{2}}\\\nonumber
&+ak\|\nabla\tau^a\|_{L^2}\|\nabla^{2}u^a\|_{L^{2}}
+k\|\nabla^{1+2\beta}\tau^a\|_{L^{2}}\|\nabla^{2}u^a\|_{L^{2}}+k\|u^a\|_{L^{\infty}}\|\nabla^{2} u^a\|_{L^{2}}\|\Delta \tau^a\|_{L^{2}}\\\nonumber
&+k\|\nabla^{2}\tau^a\|_{L^{2}}\|\nabla^{2}u^a\|_{L^{2}}\\\nonumber
\leq &\ C\delta \bar{D}_{1}+kC\|\nabla^{\beta+1}\tau^a\|_{H^{s-1}}^{2}+\frac{k}{100}\|\nabla^{2}u^a\|_{L^{2}}^{2}+Cka^{2}\|\nabla\tau^a\|_{L^2}^{2},\nonumber
\end{align}
Along the same line to the proof of (\ref{3213}), we find
\begin{align}\label{413}
&\frac{1}{2}\frac{d}{dt}\|\Lambda^{s}(u^a,\tau^a)\|_{L^{2}}^{2}+\|\Lambda^{s+\beta}\tau^a\|_{L^{2}}^{2}+a\|\Lambda^{s}\tau^a\|_{L^{2}}^{2}\\\nonumber
=&-\langle[ \Lambda^{s},u^a\cdot\nabla]u^a,\Lambda^{s}u^a\rangle-\langle\Lambda^{s}(u^a\cdot\nabla\tau^a+Q(\nabla u^a,\tau^a)),\Lambda^{s}\tau^a\rangle\\ \nonumber
\lesssim & \|u^a\|_{H^{s}}\|\Lambda^{s-\beta+1} u^a\|_{L^{2}}^{2}+\|\tau^a\|_{L^{\infty}}\|\Lambda^{s+\beta}\tau^a\|_{L^{2}}\|\Lambda^{s-\beta+1}u^a\|_{L^{2}}\\ \nonumber &+\|u^a\|_{L^{\infty}}\|\Lambda^{s+\beta}\tau^a\|_{L^{2}}\|\Lambda^{s-\beta+1}\tau^a\|_{L^{2}}+\|u^a\|_{H^{s}}\|\Lambda^{\beta+1}\tau^a\|_{H^{s-1}}^{2}\\ \nonumber
\lesssim &\ \delta \bar{D}_{1},\nonumber
\end{align}
Thanks to $s\ge1+\beta$ again, by proceeding as in the estimates of (\ref{3214}), we deduce that
\begin{align}\label{414}
&\frac{d}{dt}\langle\Lambda^{s-\beta}\tau^a,-k\nabla \Lambda^{s-\beta}u^a\rangle+\frac{k}{2}\|\nabla\Lambda^{s-\beta}u^a\|_{L^{2}}^{2}\\ \nonumber =&\langle\Lambda^{s-\beta}\left(u^a\cdot\nabla\tau^a+Q(\nabla u^a,\tau^a)+(-\Delta)^{\beta} \tau^a+a\tau^a\right),k\nabla\Lambda^{s-\beta}u^a\rangle\\ \nonumber &-\langle\Lambda^{s-\beta}\mathbb{P}\left(u^a\cdot\nabla u^a-\text{div }\tau^a\right),k\text{div }\Lambda^{s-\beta}\tau^a\rangle\\ \nonumber
\leq&\ kC\|\Lambda^{s-\beta+1}u^a\|_{L^{2}}(\|u^a\|_{L^{\infty}}\|\Lambda^{s-\beta+1} \tau^a\|_{L^{2}}+\|\tau^a\|_{L^{\infty}}\|\Lambda^{s-\beta+1}u^a\|_{L^{2}}\\\nonumber
&+\|\Lambda^{\beta+1}\tau^a\|_{H^{s-1}}\|u^a\|_{H^{s}}+\|\Lambda^{s+\beta} \tau^a\|_{L^{2}})\\ \nonumber 
&+kC\|\Lambda^{s-\beta+1}\tau^a\|_{L^{2}}^{2}+Ca^{2}\|\nabla\tau^a\|_{H^{s-1}}^{2}+\frac{k}{100}\|\Lambda^{s-\beta+1}u^a\|_{L^{2}}^{2}\\ \nonumber
\leq\ &C\delta \bar{D}_{1}+C\|\Lambda^{\beta+1}\tau^a\|_{H^{s-1}}^{2}+kCa^{2}\|\nabla\tau^a\|_{H^{s-1}}^{2}+\frac{k}{100}\|\Lambda^{s-\beta+1}u^a\|_{L^{2}}^{2}.\nonumber
\end{align}
Combining (\ref{411})-(\ref{414}) and taking $k$ small enough, we obtain
\begin{align}\label{415}
  \frac{d}{dt}\bar{E}_{1}+\bar{D}_{1}\leq 0,
\end{align}
which implies that
\begin{align*}
\frac{d}{dt}[f(t)\bar{E}_{1}]+C_{2}f^{\prime}(t)\left(\frac{k}{2}\|\nabla u^a\|_{H^{s-1}}^{2}+\|\nabla\tau^a\|_{H^{s-1}}^{2}\right)&\lesssim f^{\prime}(t)\int_{S_{0}(t)}|\xi|^{2}(| \widehat{u^a}(\xi)|^{2}+|\widehat{\tau^a}(\xi)|^{2})d\xi\\\nonumber
&\lesssim (1+t)^{-2}\ln^{-l+1}(e+t).
\end{align*}
This ensures that 
\begin{align}\label{416}
\bar{E}_{1}\leq C(1+t)^{-1}\ln^{-l}(e+t).
\end{align}

\textbf{Step3: improve inital time decay rate}

Define $S(t)=\{\xi:|\xi|^{2}\leq C_{2}(1+t)^{-1}\}$ with $C_{2}$ large enough. By (\ref{402}), we infer that
\begin{align}\label{417}
\frac{d}{dt}\bar{E}_{0}(t)+\frac{kC_{2}}{2(1+t)}\|u^a\|_{H^{s}}^{2}+\frac{C_{2}}{1+t} \|\tau^a\|_{H^{s}}^{2}\lesssim \frac{1}{1+t}\int_{S(t)}|\widehat{u^a}(\xi)|^{2}+|\widehat{\tau^a} (\xi)|^{2}d\xi.
\end{align}
Integrating (\ref{404}) over $S(t)$ with $\xi$, then we derive
\begin{align}\label{L405}
\int_{S(t)}|\widehat{u^a}|^{2}+|\widehat{\tau^a}|^{2}d\xi\lesssim \int_{S(t)}|\widehat{u}_{0}|^{2}+|\widehat{\tau}_{0}|^{2}d\xi+\int_{S(t)}\int_{0}^{t}|\widehat{F}\cdot\overline{\widehat{u^a}}|+|\widehat{G}\cdot\overline{\widehat{\tau^a}}|dt^{\prime}d\xi.   
\end{align}
Similar to the derivation of (\ref{326}) and (\ref{327}), one can deduce that
\begin{align}\label{418}
\int_{S(t)}(|\hat{u}_{0}|^{2}+|\hat{\tau}_{0}|^{2})d\xi 
 & \lesssim (1+t)^{-1}\|(u_{0},\tau_{0})\|_{\dot{B}_{2,\infty}^{-1}}^{2},
\end{align}
and
\begin{align*}
\int_{S(t)}\int_{0}^{t}|\widehat{F}\cdot\overline{\widehat{u^a}}|+|\widehat{G}\cdot\overline{\widehat{\tau^a}}|dt^{\prime}d\xi
\lesssim &\ (1+t)^{-\frac{1}{2}}\int_{0}^{t}\|F\|_{L^1}\|u^a\|_{L^2}+\|G\|_{L^1}\|\tau^a\|_{L^2}dt^{\prime}\\\nonumber
\lesssim &\ (1+t)^{-\frac{1}{2}}\int_{0}^{t}\left(\|u^a\|_{L^2}^2+\|\tau^a\|_{L^2}^2\right)\left(\|\nabla u^a\|_{L^2}+\|\nabla\tau^a\|_{L^2}\right)dt^{\prime}.\nonumber
\end{align*}
Combining the above inequalities, (\ref{L405}) and (\ref{417}), we find
\begin{align*}
&\frac{d}{dt}\bar{E}_{0}(t)+\frac{kC_{2}}{2(1+t)}\|u^a\|_{H^ {s}}^{2}+\frac{C_{2}}{1+t}\|\tau^a\|_{H^{s}}^{2}\\ \nonumber &\lesssim \frac{1}{1+t}\left[(1+t)^{-1}+(1+t)^{-\frac{1}{2}}\int_{0}^{t}(\| u^a\|_{L^{2}}^{2}+\|\tau^a\|_{L^{2}}^{2})(\|\nabla u^a\|_{L^{2}}+\|\nabla\tau^a\|_{L^{2}}) dt^{\prime}\right],\nonumber
\end{align*}
which implies that
\begin{align*}
 (1+t)^{\frac{3}{2}}\bar{E}_{0}(t)\lesssim (1+t)^{\frac{1}{2}}+(1+t)\int_{0}^{t}\left(\|u^a\|_ {L^{2}}^{2}+\|\tau^a\|_{L^{2}}^{2}\right)\left(\|\nabla u^a\|_{L^{2}}+\|\nabla\tau^a\|_{L^{2}}\right) dt^{\prime}.
\end{align*}
Let $\displaystyle N(t)=\sup_{0\leq t^{\prime}\leq t}(1+t^{ \prime})^{\frac{1}{2}}\bar{E}_{0}(t^{\prime})$. Then we find
\begin{align*}
N(t)\leq C+C\int_{0}^{t}N(t^{\prime})(1+t^{\prime})^{-\frac{1}{2}}\left(\|\nabla u^a\|_{L^{2} }+\|\nabla\tau^a\|_{L^{2}}\right)dt^{\prime}.
\end{align*}
Applying Gronwall's inequality and taking $l=4$ in (\ref{410}), we get $N(t)<C$. We duduce that
\begin{align}\label{419}
\bar{E}_{0}\lesssim (1+t)^{-\frac{1}{2}}.
\end{align}
Due to (\ref{415}), we can derive that
\begin{align}\label{420}
\frac{d}{dt}\bar{E}_{1}(t)+\frac{kC_{2}}{2(1+t)}\|\nabla u^a\|_{H^{s-1}}^{2}+\frac{C_{2}}{1+t} \|\nabla\tau^a\|_{H^{s-1}}^{2}&\lesssim \frac{1}{1+t}\int_{S(t)}|\xi|^{2}(|\widehat{u}(\xi)|^{2}+|\widehat{\tau} (\xi)|^{2})d\xi.
\end{align}
Moreover, inserting  (\ref{419}) into (\ref{420}), we get
\begin{align}\label{421}
\bar{E}_{1}\lesssim (1+t)^{-\frac{3}{2}}.
\end{align}

\textbf{Step4: improve time decay rate agian}

Applying $\dot\Delta_{j}$ to (\ref{eq2}), we find
\begin{align*}
 \left\{\begin{array}{l}\dot{\Delta}_{{j}}u_{t}+ \nabla\dot{\Delta}_{j}P-\text{div}\ \dot{\Delta}_{j}\tau=\dot{\Delta}_{j}F,\\ [1ex] \dot{\Delta}_{j}\tau_{t}+a\dot{\Delta}_{j}\tau+(-\Delta)^{\beta}\dot{\Delta}_{j}\tau-\dot{\Delta}_{j}D(u)=\dot{\Delta}_{j}G.\end{array}\right.
\end{align*}
where $G=-Q(\nabla u^a,\tau^a)-u^a\cdot\nabla\tau^a$ and $F=-u^a\cdot\nabla u^a$. Then we get
\begin{align}\label{422}
\frac{d}{dt}(\|\dot{\Delta}_{j}u^a\|_{L^{2}}^{2}+\|\dot{\Delta}_{j}\tau^a\|_{L^{2} }^{2})+2\|\Lambda^{\beta}\dot{\Delta}_{j}\tau^a\|_{L^{2}}^{2}\lesssim \|\dot{\Delta}_{j}F\| _{L^{2}}\|\dot{\Delta}_{j}u^a\|_{L^{2}}+\|\dot{\Delta}_{j}G\|_{L^{2}}\|\dot{ \Delta}_{j}\tau^a\|_{L^{2}}.
\end{align}
Multiplying (\ref{422}) by $2^{-j}$ and taking $l^{\infty}$ norm, we have
\begin{align}\label{423}
\frac{d}{dt}\left(\|u^a\|_{\dot{B}^{-\frac{1}{2}}_{2,\infty}}^{2}+\|\tau^a\|_{\dot{B}^{-\frac{1}{2}}_{2,\infty}}^{2}\right)\lesssim \|F\|_{\dot{B}^{-\frac{1}{2}}_{2,\infty}}\|u^a\|_{ \dot{B}^{-\frac{1}{2}}_{2,\infty}}+\|G\|_{\dot{B}^{-\frac{1}{2}}_{2,\infty}}\|\tau^a\|_{ \dot{B}^{-\frac{1}{2}}_{2,\infty}}.  
\end{align}
Let $\displaystyle M(t)=\sup_{0\le t^{\prime}\le t}\left(\|u^a\|_{\dot{B}^{-\frac{1}{2}}_{2,\infty}}+\|\tau^a\|_{\dot{B}^{-\frac{1}{2}}_{2,\infty}}\right)$, then we obtain
\begin{align}\label{424}
M^{2}(t)\leq CM^{2}(0)+CM(t)\int_{0}^{t}\left(\|F\|_{\dot{B}^{-\frac{1}{2}}_{2,\infty}}+\|G\|_{\dot{B}^{-\frac{1}{2}}_{2,\infty}}\right)dt^{\prime}. 
\end{align}
Using the fact $L^{\frac{4}{3}}\hookrightarrow\dot{B}^{-\frac{1}{2}}_{2,\infty}$, we deduce for any $t> 0$, that
\begin{align*}
\int_{0}^{t}(\|F\|_{\dot{B}^{-\frac{1}{2}}_{2,\infty}}+\|G\|_{\dot{B}^{-\frac{1}{2}}_{2,\infty}})dt^{\prime}&\lesssim  \int_{0}^{t}(\|F\|_{L^{\frac{4}{3}}}+\|G\|_{L^{\frac{4}{3}}})dt^{\prime}\\\nonumber
&\lesssim  \int_{0}^{t}(\|u^a\|_{L^4}+\|\tau\|_{L^4})(\|\nabla u^a\|_{L^2}+\|\nabla\tau^a\|_{L^2})dt^{\prime}\\\nonumber
&\lesssim  \int_{0}^{t}(1+t^{\prime})^{-\frac{5}{4}}dt^{\prime}<\infty.\nonumber
\end{align*}
Hence, applying gronwall's inequality, we deduce $M(t)<C$. Due to (\ref{419}) and (\ref{421}), we have
\begin{align*}
\int_{S_{1}(t)}\int_{0}^{t}|\widehat{F}\cdot\overline{\widehat{u^a}}|+|\widehat{G}\cdot\overline{\widehat{\tau^a}}|dt^{\prime}d\xi
\lesssim & \int_{0}^{t}\left(\|F\|_{L^\frac{4}{3}}\int_{S_{1}(t)}|{\widehat{u^a}}|d\xi+\|G\|_{L^\frac{4}{3}}\int_{S_{1}(t)}|{\widehat{\tau^a}}|d\xi\right
) dt^{\prime} \\
\lesssim &\ (1+t)^{-\frac{3}{4}}M(t)\ln(1+t)\\
\lesssim &\ (1+t)^{-\frac{5}{8}}.
\end{align*}
This together with (\ref{417}), (\ref{418}) and (\ref{420}) ensures that 
\begin{align}\label{427}
\bar{E}_{0}(t) \lesssim (1+t)^{-\frac{5}{8}},~~~\bar{E}_{1}(t) \lesssim (1+t)^{-\frac{13}{8}}.
\end{align} 
Multiplying (\ref{422}) by $2^{-2j}$ and taking $l^{\infty}$ norm, we achieve
\begin{align*}
\frac{d}{dt}\left(\|u^a\|_{\dot{B}^{-1}_{2,\infty}}^{2}+\|\tau^a\|_{\dot{B}^{-1}_{2,\infty}}^{2}\right)\lesssim \|F\|_{\dot{B}^{-1}_{2,\infty}}\|u^a\|_{ \dot{B}^{-1}_{2,\infty}}+\|G\|_{\dot{B}^{-1}_{2,\infty}}\|\tau^a\|_{ \dot{B}^{-1}_{2,\infty}}.  
\end{align*}
Let $\displaystyle M_1(t)=\sup_{0\le t^{\prime}\le t}\left(\|u^a\|_{\dot{B}^{-1}_{2,\infty}}+\|\tau^a\|_{\dot{B}^{-1}_{2,\infty}}\right)$, then we yield
\begin{align*}
M_1^{2}(t)\leq CM_1^{2}(0)+CM_1(t)\int_{0}^{t}\left(\|F\|_{\dot{B}^{-1}_{2,\infty}}+\|G\|_{\dot{B}^{-1}_{2,\infty}}\right)dt^{\prime}. 
\end{align*}
For any $t>0$, using (\ref{427}) and the fact $L^{1}\hookrightarrow \dot{B}^{-1}_{2,\infty}$, we arrive at
\begin{align*}
\int_{0}^{t}(\|F\|_{\dot{B}^{-1}_{2,\infty}}+\|G\|_{\dot{B}^{-1}_{2,\infty}})dt^{\prime}&\lesssim  \int_{0}^{t}(\|F\|_{L^{1}}+\|G\|_{L^{1}})dt^{\prime}\\\nonumber
&\lesssim  \int_{0}^{t}(\|u^a\|_{L^2}+\|\tau\|_{L^2})(\|\nabla u^a\|_{L^2}+\|\nabla\tau^a\|_{L^2})dt^{\prime}\\\nonumber
&\lesssim  \int_{0}^{t}(1+t^{\prime})^{-\frac{9}{8}}dt^{\prime}<\infty.\nonumber
\end{align*}
By Gronwall's inequality again, we know $M_1(t)<C$. Using (\ref{427}), we have
\begin{align}\label{428}
\int_{S_{1}(t)}\int_{0}^{t}|\widehat{F}\cdot\bar{\widehat{u}}|+|\widehat{G}\cdot\bar{\widehat{\tau}}|dt^{\prime}d\xi
\lesssim &\ (1+t)^{-1}M(t)\int_{0}^{t}(1+t)^{-\frac{9}{8}}dt^{\prime}\\\nonumber
\lesssim &\ (1+t)^{-1},\nonumber
\end{align}
This together with (\ref{417}), (\ref{L405}) and (\ref{418}) ensures that
\begin{align*}
\frac{d}{dt}\bar{E}_{0}(t)+\frac{kC_{2}}{2(1+t)}\|u^a\|_{H^{s} }^{2}+\frac{C_{2}}{1+t}\|\tau^a\|_{H^{s}}^{2} \lesssim (1+t)^{-2}.
\end{align*} 
Similarly, we deduce that
\begin{align}\label{425}
\bar{E}_{0}(t) \lesssim (1+t)^{-1},~~~\bar{E}_{1}(t) \lesssim (1+t)^{-2}.
\end{align} 
Using interpolation between Sobolev spaces, we finish the proof of Proposition \ref{4prop1}.
\end{proof}

Next we consider the integrability of $\|\nabla u^a\|_{L^\infty}$ and $\|\nabla \tau^a\|_{L^\infty}$ which is useful to establish the rate of uniform vanishing damping limit for (\ref{eq2}).
\begin{prop}\label{4prop4}
  Under the condition of Proposition \ref{4prop1}, if $s>1+\beta$, then we have
  \begin{align*}
  \int_{0}^{+\infty}\|\nabla (u^a,\tau^a)\|_{L^\infty}dt < C.
  \end{align*}
\end{prop}
\begin{proof}
We only prove 
\begin{align*}
  \int_{0}^{+\infty}\|\nabla u^a\|_{L^\infty}dt < C.
  \end{align*}
  The time integrability of $\|\nabla \tau^a\|_{L^\infty}$ can be obtained by similar argument.
The case with $s>2$ can be easily derived from Proposition \ref{4prop1} and Sobolev embedding. Next, we focus on the case where $1+\beta<s\le 2$. Applying Lemma \ref{lemma3} leads to  
\begin{align}\label{441}
\|\nabla u^a\|_{L^{\infty}}\lesssim   \|u^a\|_{L^2}^{\frac{s-\beta-1}{s-\beta+1}}\|\Lambda^{s-\beta+1}u^a\|_{L^2}^{\frac{2}{s-\beta+1}}.
\end{align}
According to (\ref{415}) and Proposition \ref{4prop1}, for $0<\delta<2$, we get
\begin{align*}
  (1+t)^{\delta}\bar{E}_1+\int_{0}^{t}(1+t^{\prime})^{\delta}\bar{D}_{1}dt^{\prime}\le C+\delta\int_{0}^{t}(1+t^{\prime})^{\delta-1}\bar{E}_{1}dt^{\prime}\le C,
\end{align*}
which in particular implies that 
\begin{align}\label{442}
 \int_{0}^{t}(1+t)^{\delta}\|\Lambda^{s-\beta+1}u^a\|_{L^2}^{2}dt^{\prime}\le C.
\end{align}
Using (\ref{441}) and (\ref{442}) and taking $b(s-\beta+1)=\delta$, we get
  \begin{align}\label{443}
  \int_{0}^{+\infty}\|\nabla u^a\|_{L^\infty}dt & \lesssim  \int_{0}^{+\infty}\|u^a\|_{L^2}^{\frac{s-\beta-1}{s-\beta+1}}\|\Lambda^{s-\beta+1}u^a\|_{L^2}^{\frac{2}{s-\beta+1}}dt\\ \nonumber
  & \lesssim  \int_{0}^{+\infty}(1+t)^{-\frac{b(s-\beta+1)}{s-\beta}}\|u^a\|_{L^2}^{\frac{s-\beta-1}{s-\beta}}dt+ \int_{0}^{+\infty}(1+t)^{b(s-\beta+1)}\|\Lambda^{s-\beta+1}u^a\|_{L^2}^{2}dt.\\ \nonumber
&\lesssim \int_{0}^{+\infty}(1+t)^{-\frac{s-\beta-1+2\delta}{2s-2\beta}}dt+C\lesssim C. \nonumber
  \end{align}
\end{proof}
 \textbf{Proof of Theorem \ref{1theo3}}
  \begin{proof}
Combining Propositions \ref{4prop1}-\ref{4prop4}, we achieve the proof of Theorem \ref{1theo3}.
  \end{proof}
\subsection{Uniform vanishing damping limit}

By virtue of the time decay rate from Proposition \ref{4prop1} and key integrability from Proposition \ref{4prop4}, we finally
obtain the uniform vanishing damping limit.
\begin{prop}\label{4prop5}
  Under the conditions in Proposition \ref{4prop4}, if $ \frac{1}{2}<\beta<1$, for any $\alpha\in [0,1]$, there holds
  \begin{align*}
  \left\|\nabla^{\alpha}\left(u^{a}-u^{0},\tau^{a}-\tau^{0}\right)\right\|_{L^{\infty}([0,\infty),L^{2})}\leq Ca ^{\frac{\beta(1+\alpha)}{\alpha\beta+3\beta-1}}.
  \end{align*} 
\end{prop}
\begin{proof}
Let $(u^a,\tau^a)$ be a solution of (\ref{eq2}) with the initial data $(u_0, \tau_0)$ and $a\in [0, 1]$. Then we get
\begin{align}\label{451}
\left\{\begin{array}{l}\partial_{t}(u^{a}-u^{0})+u^{a}\cdot\nabla(u^{a}-u^{0})+ \nabla(P_{u^{a}}-P_{u^{0}}) \\[1ex] = \mathrm{div} (\tau^{a}-\tau^{0})-(u^{a}-u^{0})\cdot\nabla u^{0},\ \ \ \ \mathrm{div} (u^{a}-u^{0})=0,\\ [1ex]
\partial_{t}(\tau^{a}-\tau^{0})+u^{a}\cdot\nabla(\tau^{a}-\tau^{0})+(-\Delta)^{\beta}( \tau^{a}-\tau^{0})+Q(\nabla u^{a},\tau^{a}-\tau^{0})\\ [1ex] =-a\tau^{a}+D(u^{a}-u^{0})-(u^{a}-u^{0})\cdot\nabla\tau^{0}-Q(\nabla(u^{a}-u^{0 }),\tau^{0}),\\[1ex] (u^{a}-u^{0})|_{t=0}=0,\ (\tau^{a}-\tau^{0})|_{t=0}=0.\end{array}\right.
\end{align}
We denote
\begin{align*}
\phi (t)=\left\|(u^{a}-u^{0},\tau^{a}-\tau^{0})\right\|_{L^{2}}+\|\nabla\left(u^{a}-u^{0},\tau^{a}-\tau^{0}\right)\|_{L^{2}},
\end{align*}
and
\begin{align*}
\varphi(t)=&\|\nabla u^0\|_{L^{\infty}}+\|\nabla u^{a}\|_{L^{\infty}}+\|\nabla\tau^0\|_{L^{\infty}}+\|\tau^{0}\|_{L^\infty}^{2}+\|\nabla u^a\|_{L^\infty}^{2}+\|\nabla u^{0}\|_{L^\infty}^{2}+\|\nabla\tau^{0}\|_{L^\infty}^{2}\\\nonumber
&+\|\nabla^{2-\beta}u^{a}\|_{L^2}^{2}+\|\nabla^{2-\beta}\tau^{0}\|_{L^2}^{2}+\|\nabla^{2-\beta}u^{0}\|_{L^2}^{2}+\|\nabla^{1+\beta} u^0\|_{L^2}^{2}.
\end{align*}
According to (\ref{451}), we deduce that
\begin{align}\label{454}
&\frac{1}{2}\frac{d}{dt}\left\|(u^{a}-u^{0},\tau^{a}- \tau^{0})\right\|_{L^{2}}^{2}+\|\Lambda^{\beta}(\tau^{a}-\tau^{0})\|_{L^{2}}^{2}+a\|\tau^{a}-\tau^{0}\|_{L^{2}}^{2} \\ \nonumber
\leq &\|\nabla u^0\|_{L^{\infty}}\phi^{2}(t)+a(1+t)^{-\frac{1}{2\beta}}\phi(t)+\|\nabla u^{a}\|_{L^{\infty}}\phi^{2}(t)\\\nonumber
&+\|\tau^0\|_{L^{\infty}}^{2}\phi^{2}(t)+\frac{1}{100}\|\nabla(u^{a}-u^{0})\|_{L^{2}}^{2}+\|\nabla\tau^0\|_{L^{\infty}}\phi^{2}(t)\\\nonumber
\leq &\ C\varphi(t) \phi^{2}(t)+C(1+t)^{-\frac{1}{2\beta}}\phi(t)+\frac{k}{100}\|\nabla(u^{a}-u^{0})\|_{L^{2}}^{2}.\nonumber
\end{align}
Using (\ref{451}) again,  we find
\begin{align*}
&\frac{1}{2}\frac{d}{dt}\left\|\nabla(u^{a}-u^{0},\tau^{a}- \tau^{0})\right\|_{L^{2}}^{2}+\|\nabla\Lambda^{\beta}(\tau^{a}-\tau^{0})\|_{L^{2}}^{2}+a\|\nabla(\tau^{a}-\tau^{0})\|_{L^{2}}^{2}\\ 
=& -\langle\nabla(u^{a}-u^{0})\cdot\nabla u^{0},\nabla(u^{a}-u^{0})\rangle-\langle a\nabla\tau^{0},\nabla (\tau^{a}-\tau^ {0})\rangle-\langle \nabla Q(\nabla u^{a},\tau^{a}-\tau^{0}),\nabla(\tau^{a}-\tau^ {0})\rangle\\
&-\langle \nabla Q(\nabla(u^{a}-u^{0}),\tau^{0}),\nabla(\tau^{a}-\tau^ {0})\rangle-\langle\nabla((u^{a}-u^{0})\cdot\nabla\tau^{0}),\nabla(\tau^{a}-\tau^ {0})\rangle\\
&-\langle(\nabla(u^{a}\cdot\nabla(\tau^{a}-\tau^{0})),\nabla(\tau^{a}-\tau^ {0})\rangle\\
\triangleq&\sum_{i=1}^{6}K_{i}
\end{align*}
By Lemma \ref{lemma5} and using the fact $s>1+\beta$, we find
\begin{align*}
K_1
\lesssim\ &\|\nabla^{2-\beta}(u^{a}-u^{0})\|_{L^2}\|\nabla^{\beta}(u^{a}-u^{0})\|_{L^2}\|\nabla u^0\|_{L^\infty}\\ \nonumber
&+\|\nabla^{2-\beta}(u^{a}-u^{0})\|_{L^2}\|u^{a}-u^{0}\|_{L^\infty}\|\nabla^{1+\beta} u^0\|_{L^2}\\ \nonumber
\le\ & \frac{k}{100}\|\nabla^{2-\beta}(u^{a}-u^{0})\|_{L^2}^{2}+C\|\nabla^{\beta}(u^{a}-u^{0})\|_{L^2}^{2}\|\nabla u^0\|_{L^\infty}^{2}\\\nonumber
&+C\delta\|\nabla^{2-\beta}(u^{a}-u^{0})\|_{L^{2}}^{2}+C\|\nabla^{1+\beta} u^0\|_{L^2}^{2}\|u^{a}-u^{0}\|_{L^2}^{2}\\\nonumber
\le\ & C\varphi(t)\phi^{2}(t)+\left(C\delta+\frac{k}{100}\right)\|\Lambda^{2-\beta}(u^{a}-u^{0})\|_{L^{2}}^{2},
\end{align*}
where $k$ to be determined below. By Proposition \ref{4prop1}, we get
\begin{align*}
K_2\le a\|\nabla\tau^0\|_{L^{2}}\|\nabla(\tau^{a}-\tau^{0})\|_{L^{2}}  \le aC(1+t)^{-\frac{1}{\beta}}\phi(t).
\end{align*}
By Lemma \ref{lemma5}, we obtain
\begin{align*}
K_3
\lesssim\ & \|\Lambda^{1+\beta}(\tau^{a}-\tau^{0})\|_{L^{2}}\left(\|\Lambda^{2-\beta}(u^{a}-u^{0})\|_{L^{2}}\|\tau^0\|_{L^{\infty}}+\|\Lambda^{1-\beta}\tau^0\|_{L^{\frac{2}{1-\beta}}}\|\nabla(u^a-u^0)\|_{L^{\frac{2}{\beta}}}\right)\\ \nonumber
\lesssim\ & \|\Lambda^{1+\beta}(\tau^{a}-\tau^{0})\|_{L^{2}}\left(\|\Lambda^{2-\beta}(u^{a}-u^{0})\|_{L^{2}}\|\tau^0\|_{L^{\infty}}+\|\Lambda\tau^0\|_{L^{2}}\|\Lambda^{2-\beta}(u^a-u^0)\|_{L^{2}}\right)\\ \nonumber
\lesssim\ &\delta\left(\|\Lambda^{1+\beta}(\tau^{a}-\tau^{0})\|_{L^{2}}^{2}+\|\Lambda^{2-\beta}(u^{a}-u^{0})\|_{L^{2}}^{2}\right).\nonumber
\end{align*}
By a similar derivation of $K_1$, we get
\begin{align*}
K_4
\lesssim\ &\|\nabla^{1+\beta}(\tau^{a}-\tau^{0})\|_{L^2}\|\nabla^{1-\beta}(\tau^{a}-\tau^{0})\|_{L^2}\|\nabla u^{a}\|_{L^\infty}\\ \nonumber
&+\|\tau^{a}-\tau^{0}\|_{L^\infty}\|\nabla^{1+\beta}(\tau^{a}-\tau^{0})\|_{L^2}\|\nabla^{2-\beta}u^{a}\|_{L^2}\\ \nonumber
\le\ &\frac{1}{100}\|\nabla^{1+\beta}(\tau^{a}-\tau^{0})\|_{L^2}^{2}+C\|\nabla^{1-\beta}(\tau^{a}-\tau^{0})\|_{L^2}^{2}\|\nabla u^{a}\|_{L^\infty}^{2}\\\nonumber
&+C\|\nabla^{1+\beta}(\tau^{a}-\tau^{0})\|_{L^2}^{2}\|\nabla^{2-\beta}u^{a}\|_{L^2}^{2}+C\|\nabla^{2-\beta}u^{a}\|_{L^2}^{2}\|\tau^{a}-\tau^{0}\|_{L^2}^{2}\\ \nonumber
\le\ & C\varphi(t)\phi^{2}(t)+\left(C\delta+\frac{1}{100}\right)\|\nabla^{1+\beta}(\tau^{a}-\tau^{0})\|_{L^2}^{2},
\end{align*}
and
\begin{align*}
K_5
\lesssim\ &\|\nabla^{1+\beta}(\tau^{a}-\tau^{0})\|_{L^2}\|\nabla^{1-\beta}(u^{a}-u^{0})\|_{L^2}\|\nabla \tau^{0}\|_{L^\infty}\\ \nonumber &+\|u^{a}-u^{0}\|_{L^\infty}\|\nabla^{1+\beta}(\tau^{a}-\tau^{0})\|_{L^2}\|\nabla^{2-\beta}\tau^{0}\|_{L^2}\\ \nonumber 
\le\ &\frac{1}{100}\|\nabla^{1+\beta}(\tau^{a}-\tau^{0})\|_{L^2}^{2}+C\|\nabla^{1-\beta}(u^{a}-u^{0})\|_{L^2}^{2}\|\nabla \tau^{0}\|_{L^\infty}^{2}+C\|u^{a}-u^{0}\|_{L^2}^{2}\|\nabla^{2-\beta}\tau^{0}\|_{L^2}^{2}\\ \nonumber
&+C\|\nabla^{2-\beta}\tau^{0}\|_{L^2}^{2}\|\Lambda^{2-\beta}(u^{a}-u^{0})\|_{L^2}^{2}\\\nonumber
\le\ & C\varphi(t)\phi^{2}(t)+\frac{1}{100}\|\nabla^{1+\beta}(\tau^{a}-\tau^{0})\|_{L^2}^{2}+C\delta\left(\|\nabla^{1+\beta}(\tau^{a}-\tau^{0})\|_{L^2}^2+\|\Lambda^{2-\beta}(u^{a}-u^{0})\|_{L^2}^2\right).
  \nonumber
\end{align*}
Due to the fact $\mathrm{div}~u^a=0$, we have
\begin{align*}
K_6\lesssim \|\nabla u^a\|_{L^\infty}\|\nabla(\tau^{a}-\tau^{0})\|_{L^2}^{2}\le \|\nabla u^a\|_{L^\infty}\phi^2(t).
\end{align*}
Collecting the estimates of $K_1$ through $K_6$ leads to 
\begin{align}\label{4512}
&\frac{1}{2}\frac{d}{dt}\left\|(\nabla(u^{a}-u^{0}),\nabla(\tau^{a}- \tau^{0})\right\|_{L^{2}}^{2}+\|\nabla\Lambda^{\beta}(\tau^{a}-\tau^{0})\|_{L^{2}}^{2}+a\|\nabla(\tau^{a}-\tau^{0})\|_{L^{2}}^{2}\\ \nonumber
\leq &\ C\varphi(t)\phi^{2}(t)+\left(C\delta+\frac{k}{100}\right)\|\nabla(u^{a}-u^{0})\|_{H^{1-\beta}}^{2}+\left(C\delta+\frac{1}{50}\right)\|\Lambda^{\beta}(\tau^{a}-\tau^{0})\|_{H^1}^{2}\\\nonumber
&+aC(1+t)^{-\frac{1}{\beta}}\phi(t).\nonumber
\end{align}
Along the same line to the proof of (\ref{454}), we have
\begin{align}\label{4513}
 -&\frac{d}{dt}\langle\nabla (u^{a}-u^{0}),\tau^{a}-\tau^{0}\rangle+\frac{1}{2}\|\nabla (u^a-u^0)\|_{L^2}^{2}=\langle(-\Delta)^{\beta}(\tau^a-\tau^0),\nabla(u^a-u^0)\rangle\\ \nonumber
 &+\langle a\tau^a,\nabla (u^{a}-u^{0})\rangle+\langle u^a\cdot\nabla(\tau^a-\tau^0),\nabla (u^{a}-u^{0})\rangle+\langle Q(\nabla u^a,\tau^{a}-\tau^{0}),\nabla (u^{a}-u^{0})\rangle\\\nonumber
 &+\langle (u^a-u^0)\cdot\nabla\tau^0,\nabla (u^{a}-u^{0})\rangle+\langle Q(\nabla (u^a-u^0),\tau^{0}),\nabla (u^{a}-u^{0})\rangle\\\nonumber
 &-\langle \mathbb{P}(u^a\cdot\nabla (u^{a}-u^{0})),\mathrm{div}(\tau^{a}-\tau^{0})\rangle+\langle \mathbb{P}\mathrm{div}(\tau^{a}-\tau^{0}),\mathrm{div}(\tau^{a}-\tau^{0})\rangle\\\nonumber
 &-\langle \mathbb{P}((u^a-u^0)\cdot\nabla u^{0})),\mathrm{div}(\tau^{a}-\tau^{0})\rangle \\\nonumber
 \lesssim &\ \|\Lambda^{2\beta}(\tau^{a}-\tau^{0})\|_{L^{2}}\|\nabla(u^{a}-u^{0})\|_{L^{2}}+a\|\mathrm{div}\tau^a\|_{L^2}\|u^{a}-u^{0}\|_{L^{2}}\\\nonumber
 &+\| u^{a}\|_{L^\infty}\|\nabla(u^{a}-u^{0})\|_{L^{2}}\|\nabla(\tau^{a}-\tau^{0})\|_{L^{2}}+C\|\nabla u^{a}\|_{L^\infty}\|\nabla(u^{a}-u^{0})\|_{L^{2}}\|\tau^a-\tau^0\|_{L^{2}}\\\nonumber
 &+\|\nabla\tau^{0}\|_{L^\infty}\|\nabla(u^{a}-u^{0})\|_{L^{2}}\|u^a-u^0\|_{L^{2}}+\|\nabla(u^{a}-u^{0})\|_{L^{2}}^{2}\|\tau^0\|_{L^\infty}\\\nonumber
 &+\|u^{a}\|_{L^\infty}\|\nabla(u^{a}-u^{0})\|_{L^{2}}\|\nabla(\tau^{a}-\tau^{0})\|_{L^{2}}+\|\nabla(\tau^{a}-\tau^{0})\|_{L^{2}}^{2}\\\nonumber
 &+\|\nabla u^{0}\|_{L^\infty}\|u^{a}-u^{0}\|_{L^{2}}\|\nabla(\tau^{a}-\tau^{0})\|_{L^{2}}\\\nonumber
 \le &\frac{1}{100}\|\nabla(u^{a}-u^{0})\|_{L^{2}}^{2}+C\|\Lambda^{2\beta}(\tau^{a}-\tau^{0})\|_{L^2}^{2}+C\|\nabla(\tau^{a}-\tau^{0})\|_{L^2}^{2}+Ca(1+t)^{-1}\phi(t)\\\nonumber
 &+C\varphi(t)\phi^{2}(t)+C\delta\left(\|\nabla(u^{a}-u^{0})\|_{L^{2}}^{2}+\|\nabla(\tau^{a}-\tau^{0})\|_{L^{2}}^{2}\right)\nonumber
 \end{align}
Using (\ref{451}),  we deduce that
\begin{align}\label{4514}
 &-\frac{d}{dt}\langle\Lambda^{1-\beta}\nabla (u^{a}-u^{0}),\Lambda^{1-\beta}(\tau^{a}-\tau^{0})\rangle+\frac{1}{2}\|\nabla\Lambda^{1-\beta} (u^a-u^0)\|_{L^2}^{2}\\ \nonumber
 =&\ \langle\Lambda^{1-\beta}(-\Delta)^{\beta}(\tau^a-\tau^0),\Lambda^{1-\beta}\nabla(u^a-u^0)\rangle+\langle a\Lambda^{1-\beta}\tau^a,\nabla \Lambda^{1-\beta}(u^{a}-u^{0})\rangle\\\nonumber
 &+\langle\Lambda^{1-\beta}(u^a\cdot\nabla(\tau^a-\tau^0)),\nabla\Lambda^{1-\beta} (u^{a}-u^{0})\rangle+\langle \Lambda^{1-\beta}Q(\nabla u^a,\tau^{a}-\tau^{0}),\nabla\Lambda^{1-\beta} (u^{a}-u^{0})\rangle \\\nonumber
 &+\langle \Lambda^{1-\beta}((u^a-u^0)\cdot\nabla\tau^0),\nabla\Lambda^{1-\beta} (u^{a}-u^{0})\rangle+\langle \Lambda^{1-\beta}Q(\nabla (u^a-u^0),\tau^{0}),\nabla\Lambda^{1-\beta} (u^{a}-u^{0})\rangle\\\nonumber
 &-\langle \Lambda^{1-\beta}\mathbb{P}(u^a\cdot\nabla (u^{a}-u^{0})),\Lambda^{1-\beta}\mathrm{div}(\tau^{a}-\tau^{0})\rangle
 +\langle \Lambda^{1-\beta}\mathbb{P}\mathrm{div}(\tau^{a}-\tau^{0}),\Lambda^{1-\beta}\mathrm{div}(\tau^{a}-\tau^{0})\rangle\\ \nonumber
 &-\langle \Lambda^{1-\beta}\mathbb{P}((u^a-u^0)\cdot\nabla u^{0})),\Lambda^{1-\beta}\mathrm{div}(\tau^{a}-\tau^{0})\rangle\\ \nonumber
 \triangleq&\sum_{i=1}^{9}L_{i}.
\end{align}
Then we infer that
\begin{align*}
 L_1\lesssim  \|\Lambda^{1+\beta}(\tau^{a}-\tau^{0})\|_{L^{2}}\|\Lambda^{2-\beta}(u^{a}-u^{0})\|_{L^{2}}
\le \frac{1}{100}\|\Lambda^{2-\beta}(u^{a}-u^{0})\|_{L^{2}}^{2}+C\|\Lambda^{1+\beta}(\tau^{a}-\tau^{0})\|_{L^2}^{2},\nonumber
\end{align*}
and
\begin{align*}
L_2\lesssim a\|\Lambda\tau^a\|_{L^2}\|\Lambda^{2-2\beta}(u^{a}-u^{0})\|_{L^{2}}
\lesssim a(1+t)^{-1}\phi(t). 
\end{align*}
By virtue of Lemma \ref{lemma5}, we have
\begin{align*}
L_3
\lesssim &\ \|\Lambda^{2-\beta}(u^{a}-u^{0})\|_{L^{2}}\left(\|u^a\|_{L^\infty}\|\Lambda^{2-\beta}(\tau^{a}-\tau^{0})\|_{L^{2}}+\|\nabla(\tau^{a}-\tau^{0})\|_{L^{\frac{2}{\beta}}}\|\Lambda^{1-\beta}u^{a}\|_{L^{\frac{2}{1-\beta}}}\right)\\
\lesssim &\ \|\Lambda^{2-\beta}(u^{a}-u^{0})\|_{L^{2}}\left(\|u^a\|_{L^\infty}\|\Lambda^{2-\beta}(\tau^{a}-\tau^{0})\|_{L^{2}}+\|\Lambda^{2-\beta}(\tau^{a}-\tau^{0})\|_{L^{2}}\|\Lambda u^{a}\|_{L^{2}}\right)\\
\lesssim &\ \delta (\|\nabla(u^{a}-u^{0})\|_{H^{1-\beta}}^{2}+\|\Lambda^{\beta}(\tau^{a}-\tau^{0})\|_{H^1}^{2}).
\end{align*}
By proceeding as in the estimates of $K_4$, we obtain
\begin{align*}
L_4
\lesssim &\ \|\Lambda^{2-\beta}(u^{a}-u^{0})\|_{L^{2}}\left(\|\nabla u^a\|_{L^\infty}\|\Lambda^{1-\beta}(\tau^{a}-\tau^{0})\|_{L^{2}}+\|\tau^{a}-\tau^{0}\|_{L^\infty}\|\Lambda^{2-\beta}u^{a}\|_{L^{2}}\right)\\ 
\le &\ C\varphi(t)\phi^{2}(t)+\frac{1}{100}\|\Lambda^{2-\beta}(u^{a}-u^{0})\|_{L^{2}}^{2}+C\delta\left(\|\Lambda^{\beta}(\tau^{a}-\tau^{0})\|_{H^1}^2+\|\Lambda^{2-\beta}(u^{a}-u^{0})\|_{L^{2}}^{2}\right),
\end{align*}
and 
\begin{align*}
L_5
 \lesssim &\  \|\Lambda^{2-\beta}(u^{a}-u^{0})\|_{L^{2}}(\|\tau^0\|_{L^\infty}\|\Lambda^{2-\beta}(u^{a}-u^{0})\|_{L^{2}}+\|u^{a}-u^{0}\|_{L^{\infty}}\|\Lambda^{2-\beta}\tau^{0}\|_{L^{2}})\\
 \le &\ C\delta\|\Lambda^{2-\beta}(u^{a}-u^{0})\|_{L^{2}}^{2}+\frac{1}{100}\|\Lambda^{2-\beta}(u^{a}-u^{0})\|_{L^{2}}^{2}+C\varphi(t)\phi^{2}(t). \nonumber
 \end{align*}
By a similar derivation of $L_3$, Then we infer that
 \begin{align*}
L_6
\lesssim \ \delta \|\nabla(u^{a}-u^{0})\|_{H^{1-\beta}}^{2},
~~~
L_7\lesssim  \delta \left(\|\nabla(u^{a}-u^{0})\|_{H^{1-\beta}}^{2}+\|\Lambda^{\beta}(\tau^{a}-\tau^{0})\|_{H^{1}}^{2}\right).
\end{align*}
Moreover, we have
 \begin{align*}
L_8\lesssim \|\Lambda^{\beta+1}(\tau^{a}-\tau^{0})\|_{L^{2}}^{2},
\end{align*}
and
\begin{align*}
L_9
 \le  C\delta(\|\Lambda^{2-\beta}(u^{a}-u^{0})\|_{L^{2}}^{2}+\Lambda^{\beta}(\tau^{a}-\tau^{0})\|_{H^1}^2)+\frac{1}{100}\|\Lambda^{2-\beta}(u^{a}-u^{0})\|_{L^{2}}^{2}+C\varphi(t)\phi^{2}(t). 
\end{align*}
Substituting the estimates of $L_1$ to $L_9$ into (\ref{4514}), we get
\begin{align}\label{4515}
&\frac{d}{dt}\langle\Lambda^{1-\beta}\nabla (u^{a}-u^{0}),\Lambda^{1-\beta}(\tau^{a}-\tau^{0})\rangle+\frac{1}{2}\|\nabla\Lambda^{1-\beta} (u^a-u^0)\|_{L^2}^{2}\\ \nonumber
&\le C\varphi(t)\phi^{2}(t)+Ca(1+t)^{-1}\phi(t)+\frac{1}{25}\|\Lambda^{2-\beta}(u^{a}-u^{0})\|_{L^{2}}^{2}+C\|\Lambda^{1+\beta}(\tau^{a}-\tau^{0})\|_{L^2}^{2}\\ \nonumber
&~~~
+C\delta(\|\nabla(u^{a}-u^{0})\|_{H^{1-\beta}}^{2}+\|\Lambda^{\beta}(\tau^{a}-\tau^{0})\|_{H^{1}}^{2}).
\end{align}
By summarizing (\ref{454}),(\ref{4512}),(\ref{4513}) and (\ref{4515}), for small enough constant $k>0$, it comes out
\begin{align}\label{4516}
  &\frac{d}{dt}\left[\frac{1}{2}\phi^{2}(t)-k\langle\nabla (u^{a}-u^{0}),(\tau^{a}-\tau^{0})\rangle_{H^{1-\beta}}\right]+\frac{k}{4}\|\nabla(u^{a}-u^{0})\|_{H^{1-\beta}}^{2}+\frac{1}{2}\|\Lambda^{\beta}(\tau^{a}-\tau^{0})\|_{H^{1}}^{2}\\ \nonumber
  &\le C\varphi(t)\phi^{2}(t)+a(1+t)^{-\frac{1}{2\beta}}\phi(t).
\end{align}
Using Propositions \ref{4prop1} and \ref{4prop4}, we know  $\varphi(t)\in L^{1}(0,+\infty)$. Then, we deduce from (\ref{4516}) that 
\begin{align}\label{4518}
  \phi(t)\lesssim a(1+t)^{-\frac{1}{2\beta}+1}.
\end{align}
Using (\ref{4518}), when $a^{\frac{2\beta}{\alpha\beta+3\beta-1}}t\le 1$, then we get
\begin{align*}
\phi(t)\lesssim a^{\frac{\beta(1+\alpha)}{\alpha\beta+3\beta-1}}.
\end{align*}
When $a^{\frac{2\beta}{\alpha\beta+3\beta-1}}t\ge 1$, using Proposition \ref{4prop1}, we get
\begin{align*}
\left\|\nabla^{\alpha}\left(u^{a}-u^{0},\tau^{a}-\tau^{0}\right)\right\|_{L^{\infty}([0,\infty),L^{2})}\lesssim (1+t)^{-\frac{\alpha+1}{2}}\lesssim a^{\frac{\beta(1+\alpha)}{\alpha\beta+3\beta-1}}.
\end{align*}
Then we complete the proof. 
\end{proof}

\begin{rema}
   When $\frac{1}{2}\le\beta\le1$, we find $a^{\frac{1}{2\beta}}\le a^{\frac{\beta}{3\beta-1}}\le a^{\frac{1}{2}}$. We conjecture that the sharp uniform vanishing damping rate of $L^2$ norm should be $a^{\frac{1}{2\beta}}$, where $\frac{1}{2}\le\beta\le1$.
\end{rema}

\begin{rema}\label{4reme8}
  In fact, when $\beta=\frac{1}{2}$, from (\ref{4516}), we can get
  \begin{align*}
  \phi(t)\lesssim a\log(1+t).
\end{align*}
For any $\alpha\in [0,1]$, if $a^{\frac{2}{\alpha+1}}t\leq 1$, then we get
\begin{align*}
\phi(t)\lesssim a\log(1+a^{-\frac{2}{\alpha+1}})\lesssim a^{1-\varepsilon},
\end{align*}
where $\varepsilon\in (0,1)$. If $a^{\frac{2}{\alpha+1}}t\geq 1$, using Proposition \ref{4prop1}, we have
\begin{align*}
\left\|\nabla^{\alpha}(u^{a}-u^{0},\tau^{a}-\tau^{0})\right\|_{L^{\infty}([0,\infty),L^{2})}\lesssim (1+t)^{-\frac{1+\alpha}{2}}\lesssim a.
\end{align*}
Hence, we have
  \begin{align*}
  \left\|\nabla^{\alpha}(u^{a}-u^{0},\tau^{a}-\tau^{0})\right\|_{L^{\infty}([0,\infty),L^{2})}\leq Ca ^{1-\varepsilon}.
  \end{align*} 
  \end{rema}

  \textbf{Proof of Theorem \ref{1theo5}:} \\
Combining Proposition \ref{4prop5} and Remark \ref{4reme8}, we finish the proof. \hfill$\Box$

\subsection{Time decay rates and sharp vanishing damping rates of $\mathrm{tr}\tau^a$ }

Next we introduce two useful lemmas to improve the time decay rate for $\mathrm{tr}\tau^a$. 

\begin{lemm}\label{lemma2}\cite{Bahouri2011}
		Let $\mathscr{C}$ be an annulus. Positive constants $c$ and $C$ exist such that for any $p\in[1,+\infty]$ and any couple $(t,\lambda)$ of positive real numbers, we have
		\begin{align*}
			{\rm Supp}~\hat{u} \subset \lambda\mathscr{C} \Rightarrow \|e^{-t(-\Delta)^{\beta}}u\|_{L^p} \leq Ce^{-ct\lambda^{2\beta}}\|u\|_{L^p}.
		\end{align*}
	\end{lemm}

\begin{lemm}\label{lemma7}\cite{Wu2}
Assume $0<s_{1}\leqslant s_{2}$. Then for some constant $C>0$,
\begin{align*}
\int_{0}^{t}(1+t-t^{\prime})^{-s_{1}}(1+t^{\prime})^{-s_{2}}dt^{\prime}\leqslant C\begin{cases}(1+t)^{-s_{1}},&\textit{if }s_{2}>1,\\ (1+t)^{-s_{1}}\ln(1+t),&\textit{if }s_{2}=1,\\ (1+t)^{1-s_{1}-s_{2}},&\textit{if }s_{2}<1.\end{cases}
\end{align*}
	\end{lemm}
Firstly, we discover $\tilde{\tau}^{a}=\mathrm{tr}\tau$ satisfies 
    	\begin{align}\label{eq3}
		\left\{\begin{array}{l}
			\partial_{t}\tilde{\tau}^{a}+(-\Delta)^{\beta}\tilde{\tau}^{a}+a\tilde{\tau}^{a}=Q_{1}(u ^{a},\nabla\tau^{a})+Q_{2}(\nabla u^{a},\tau^{a}),\\[1ex]
  \tilde{\tau}^{a}(0)=\tilde{\tau}_{0},\\[1ex]
		\end{array}\right.
	\end{align}
where $Q_{1}=-\text{tr}(u^{a}\cdot\nabla\tau^{a})$ and $Q_{2}=-\text{tr}Q(\nabla u^{a},\tau^{a})$. 
The linearized equation is 
\begin{align*}
 \partial_{t}\tilde{\tau}^{a}+(-\Delta)^{\beta}\tilde{\tau}^{a}+a\tilde{\tau}^{a}=0,\ \tilde{\tau}^{a}(0)=\tilde{\tau}_{0},
\end{align*}
whose solution can be written as
\begin{align*}
\tilde{\tau}^{a}(t)=e^{-at}e^{-(-\Delta)^{\beta}t}\tilde{\tau}_{0}.
\end{align*}
In the low frequencies regime $|\xi|\le1$, we have
\begin{align*}
\|e^{-at}e^{-(-\Delta)^{\beta}t}\tilde{\tau}^{a}\|_{L^2}\le e^{-|\xi|^{2\beta}t},
\end{align*}
For high frequencies $|\xi|> 1$, we have
\begin{align*}
\|e^{-at}e^{-(-\Delta)^{\beta}t}\tilde{\tau}^{a}\|_{L^\infty}\lesssim e^{-\lambda t}.
\end{align*}
where the constant $\lambda$ is independent of $a$. Then we can  improve the time decay rate of  $\|\tilde{\tau}^{a}\|_{H^1}$.

\begin{prop}\label{4prop12}
Under the conditions in Proposition \ref{4prop1}, there holds
\begin{align*}
\|\tilde{\tau}^{a}\|_{L^2}\le C(1+t)^{-\frac{1}{2\beta}},
\end{align*}
\end{prop}

\begin{proof}
Using Duhamel's formula, $\tilde{\tau}^{a}$ can be represented as 
\begin{align*}
\tilde{\tau}^{a}(t)=e^{-at}e^{-(-\Delta)^{\beta}t}\tilde{\tau}_{0}+\int_{0}^{t}e^{-a(t-t^{\prime})}e^{-(-\Delta)^{\beta}(t-t^{\prime})}[Q_{1}(u^{a},\nabla\tau^{a})+Q_{2}(\nabla u^{a},\tau^{a})]dt^{\prime}.
\end{align*} 
By Lemmas \ref{lemma1} and \ref{lemma2}, for any $t\ge 1$,  we get
\begin{align*}
\|e^{-at}e^{-(-\Delta)^{\beta}t}\tilde{\tau}_{0}\|_{L^2}&=\left(\sum_{j\in \mathbb{Z}}\| e^{-at}e^{-(-\Delta)^{\beta}t}\dot\Delta_{j}\tilde{\tau}_{0}\|_{L^2}^{2}\right)^{\frac{1}{2}}\\ \nonumber
& \lesssim \left(\sum_{j\in \mathbb{Z}}e^{-2c2^{2j\beta}t}\|\dot\Delta_{j}\tilde{\tau}_{0}\|_{L^2}^{2}\right)^{\frac{1}{2}} \\ \nonumber
&\lesssim \|\tilde{\tau}_{0}\|_{\dot B^{-1}_{2,\infty}}\left(\sum_{j\in \mathbb{Z}}e^{-2c2^{2j\beta}t}2^{2j}\right)^{\frac{1}{2}}\\ \nonumber
    &\lesssim \|\tilde{\tau}_{0}\|_{\dot B^{-1}_{2,\infty}}t^{-\frac{1}{2\beta}}.
\end{align*}
When $t<1$, we know
\begin{align*}
\|e^{-at}e^{-(-\Delta)^{\beta}t}\tilde{\tau}_{0}\|_{L^2}\lesssim \|\tilde{\tau}_{0}\|_{ L^{2}}.
\end{align*}
Combining above estimates, we deduce that
\begin{align}\label{4121}
\|e^{-at}e^{-(-\Delta)^{\beta}t}\tilde{\tau}_{0}\|_{L^2}\lesssim (1+t)^{-\frac{1}{2\beta}}.
\end{align}
For nonlinear term $Q_{1}(u^{a},\nabla\tau^{a})$, we write 
\begin{align*}
&\left\|\int_{0}^{t}e^{-a(t-t^{\prime})}e^{-(-\Delta)^{\beta}(t-t^{\prime})}[Q_{1}(u^{a},\nabla\tau^{a})]dt^{\prime}\right\|_{L^2}  \\ \nonumber
%\lesssim &\int_{0}^{t}\left\|e^{-a(t-t^{\prime})}e^{-(-\Delta)^{\beta}(t-t^{\prime})}[Q_{1}(u^{a},\nabla\tau^{a})]\right\|_{L^2}dt^{\prime} \\ \nonumber
\lesssim &\int_{0}^{t}\left(\int_{|\xi|<1}e^{-2|\xi|^{2\beta}(t-t^{\prime})}|\mathscr{F}(Q_{1}(u^{a},\nabla\tau^{a}))|^{2}d\xi\right)^{\frac{1}{2}} dt^{\prime}\\\nonumber
&+\int_{0}^{t}\left(\int_{|\xi|\ge1}e^{-2|\xi|^{2\beta}(t-t^{\prime})}|\mathscr{F}(Q_{1}(u^{a},\nabla\tau^{a}))|^{2}d\xi\right)^{\frac{1}{2}} dt^{\prime}\\ \nonumber
=&N_1+N_2.  \nonumber
\end{align*}
Using Proposition \ref{4prop1} and Lemma \ref{lemma7}, one can deduce that
\begin{align*}
N_1
&\lesssim \int_{0}^{t}(1+t-t^{\prime})^{-\frac{1}{2\beta}}\|(Q_{1}(u^{a},\nabla\tau^{a}))\|_{L^{1}}dt^{\prime}\\ \nonumber
&\lesssim \int_{0}^{t}(1+t-t^{\prime})^{-\frac{1}{2\beta}}\|u^{a}\|_{L^{2}}\|\nabla\tau^{a}\|_{L^{2}}dt^{\prime}\\ \nonumber
&\lesssim \int_{0}^{t}(1+t-t^{\prime})^{-\frac{1}{2\beta}}(1+t^{\prime})^{-\frac{3}{2}}dt^{\prime}\\ \nonumber
&\lesssim (1+t)^{-\frac{1}{2\beta}}.
\end{align*}
Due to Proposition \ref{4prop1}, Lemmas \ref{lemma2} and \ref{lemma7}, we obtain
\begin{align*}
N_2
&\lesssim \int_{0}^{t}(1+t-t^{\prime})^{-\frac{1}{2\beta}}\|(Q_{1}(u^{a},\nabla\tau^{a}))\|_{L^{2}}dt^{\prime}\\ \nonumber
&\lesssim \int_{0}^{t}(1+t-t^{\prime})^{-\frac{1}{2\beta}}\|u^{a}\|_{L^{\infty}}\|\nabla\tau^{a}\|_{L^{2}}dt^{\prime}\\ \nonumber
&\lesssim \int_{0}^{t}(1+t-t^{\prime})^{-\frac{1}{2\beta}}(1+t^{\prime})^{-\frac{3}{2}}dt^{\prime}\\ \nonumber
&\lesssim (1+t)^{-\frac{1}{2\beta}}.
\end{align*}
Therefore, we conclude that
\begin{align}\label{4122}
\left\|\int_{0}^{t}e^{-a(t-t^{\prime})}e^{-(-\Delta)^{\beta}(t-t^{\prime})}[Q_{1}(u^{a},\nabla\tau^{a})]dt^{\prime}\right\|_{L^2}\lesssim (1+t)^{-\frac{1}{2\beta}}.
\end{align}
Similarly, we have
\begin{align}\label{4123}
\left\|\int_{0}^{t}e^{-a(t-t^{\prime})}e^{-(-\Delta)^{\beta}(t-t^{\prime})}[Q_{2}(\nabla u^{a},\tau^{a})]dt^{\prime}\right\|_{L^2}\lesssim (1+t)^{-\frac{1}{2\beta}}.
\end{align}
According to (\ref{4121}), (\ref{4122}) and (\ref{4123}), we get
\begin{align}\label{4124}
\|\tilde{\tau}^{a}\|_{L^2}\lesssim (1+t)^{-\frac{1}{2\beta}}.
\end{align}

\end{proof}
Combining Proposition \ref{4prop5} and proposition \ref{4prop12}, we have 
\begin{coro}\label{4coro13}
Under the conditions in Proposition \ref{4prop5}, there holds
\begin{align*}
\|\tilde{\tau}^{a}-\tilde{\tau}^{0}\|_{L^{\infty}([0,\infty),L^{2})}\le Ca^{\frac{1}{2\beta}}.
\end{align*}
\end{coro}
\begin{proof}
According to (\ref{4518}), when $at\le 1$, then we get
\begin{align*}
\|\tilde{\tau}^{a}-\tilde{\tau}^{0}\|_{L^{2}}\lesssim a(1+t)^{-\frac{1}{2\beta}+1}\lesssim a^{\frac{1}{2\beta}}.
\end{align*}
When $at> 1$, by (\ref{4124}),
\begin{align*}
\|\tilde{\tau}^{a}-\tilde{\tau}^{0}\|_{L^{2}}\lesssim (1+t)^{-\frac{1}{2\beta}}\lesssim a^{\frac{1}{2\beta}}.
\end{align*}
Then we complete the proof. 
\end{proof}
\textbf{Proof of Theorem \ref{1theo6}:} \\
Combining Proposition \ref{4prop12} and Corollary \ref{4coro13}, we finish the proof. \hfill$\Box$

Finally, we show that the vanishing damping rate in Corollary \ref{4coro13}  is sharp.
\begin{prop}\label{4prop14}
 There exists initial data $(u_0, \tau_0)$ which satisfies the conditions in Proposition \ref{4prop5} such
that the following estimates hold: 
\begin{align*}
\|\tilde{\tau}^{a}-\tilde{\tau}^{0}\|_{L^{\infty}([0,\infty),L^{2})}\ge Ca^{ \frac{1}{2\beta}}.
\end{align*}
\end{prop}
\begin{proof}
From (\ref{eq3}), we have
\begin{align*}
  &\tilde{\tau}^{0}(t)-\tilde{\tau}^{a}(t)\\ \nonumber
  =& \left(1-e^{-at}\right)e^{-(-\Delta)^{\beta} t}\tilde{\tau}_{0}+\int_{0}^{t}\left(1-e^{-a\left(t-t^{\prime}\right)}\right)e^{ -(-\Delta)^{\beta}(t-t^{\prime})}[Q_{1}(u^{a},\nabla\tau^{a})+Q_{2}(\nabla u^{a},\tau^{a})]dt^{\prime}\\ \nonumber
& +\int_{0}^{t}e^{-(-\Delta)^{\beta}\left(t-t^{\prime}\right)}\left[Q_{1}(u^{0},\nabla\tau^{0})-Q_{1}(u^ {a},\nabla\tau^{a})+Q_{2}(\nabla u^{0},\tau^{0})-Q_{2}(\nabla u^{a},\tau^{a})\right]dt^{\prime}.\nonumber
\end{align*}
Here, we take $u_0 = \varepsilon\varphi$, $\tau_0 = \varepsilon\psi$ with $\widehat{\varphi}, \widehat{\psi}\in C_{0}^{\infty}$ and $\varepsilon$ sufficiently small. Besides, we assume that $\widehat{\psi}(\xi) = 1$ for $|\xi|\le 1$. For $ t\ge 1$, we have
\begin{align*}
\|(1-e^{-at})e^{-(-\Delta)^{\beta} t}\tilde{\tau}_{0}\|_{L^{2}}& =\|(1-e^{-at})e^{-|\xi|^{2\beta}t}\varepsilon\widehat{tr\psi}(\xi)\|_{L^{2}}\\ \nonumber &\gtrsim \varepsilon(1-e^{-at})\|e^{-|\xi|^{2\beta}t}\|_{L^{ 2}(|\xi|\leq 1)}\\ \nonumber
 &\gtrsim \varepsilon(1-e^{-at})t^{-\frac{1}{2\beta}},\nonumber
\end{align*}
which leads to
\begin{align*}
\|(1-e^{-at})e^{-(-\Delta)^{\beta} t}\tilde{\tau}_{0}\|_{L^{\infty}(0,\infty;L^{2})} \gtrsim\|\varepsilon(1-e^{-at})t^{-\frac{1}{2\beta}}\|_{L^{\infty}_{t}}\gtrsim \varepsilon a^{\frac{1}{2\beta}}.
\end{align*}
Hence, we have
 \begin{align*}
 &\quad\left\|\int_{0}^{t}\left(1-e^{-a(t-t^{\prime})}\right)e^{-(-\Delta)^{\beta}(t-t^{\prime})}[Q _{1}+Q_{2}]dt^{\prime}\right\|_{L^{2}}\\ \nonumber
 &\lesssim \int_{0}^{t}\min\{1,a(t-t^{\prime}) \}(t-t^{\prime})^{-\frac{1}{2\beta}}\|Q_{1}+Q_{2}\|_{L^{1}}dt^{\prime}\\ \nonumber
 &\lesssim\int_{0}^{t}a^{\frac{1}{2\beta}}\varepsilon^{2}(1+t^{\prime})^{-\frac {3}{2}}dt^{\prime}\lesssim\varepsilon^{2}a^{\frac{1}{2\beta}}.
 \end{align*}
Due to $\frac{1}{2}<\beta<1$, we deduce from Propositions \ref{4prop1}, \ref{4prop5} and Lemma \ref{lemma7} that
 \begin{align*}
& \left\|\int_{0}^{t} e^{-(-\Delta)^{\beta}(t-t^{\prime})}\left[Q_{1}\left(u^{0}, \nabla \tau^{0}\right)-Q_{1}\left(u^{a}, \nabla \tau^{a}\right)\right] d t^{\prime}\right\|_{L^{2}} \\\nonumber
%\lesssim & \int_{0}^{t}\left\|e^{-|\xi|^{2\beta}(t-t^{\prime})}\mathscr{F}\left[Q_{1}\left(u^{0}, \nabla \tau^{0}\right)-Q_{1}\left(u^{a}, \nabla \tau^{a}\right)\right]\right\|_{L^{2}}dt^{\prime}\\\nonumber
\lesssim & \int_{0}^{t}\|e^{-|\xi|^{2\beta}(t-t^{\prime})}\|_{L^{\frac{2\beta}{2\beta-1}}}\left\|u^{0}-u^{a}\right\|_{L^{\frac{2\beta}{2\beta-1}}}\left\|\nabla \tau^{0}\right\|_{L^{2}}+\|e^{-|\xi|^{2\beta}(t-t^{\prime})}\|_{L^{2}}\left\|u^{a}\right\|_{L^{2}}\left\|\nabla (\tau^{0}- \tau^{a})\right\|_{L^{2}}dt^{\prime}\\\nonumber
\lesssim & \int_{0}^{t}(t-t^{\prime})^{-\frac{2\beta-1}{2\beta^{2}}}\left\|\nabla ^{\frac{1}{\beta}-1}(u^{0}-u^{a})\right\|_{L^{2}}\left\|\nabla \tau^{0}\right\|_{L^{2}}+(t-t^{\prime})^{-\frac{1}{2\beta}}\left\|u^{a}\right\|_{L^{2}}\left\|\nabla (\tau^{0}-\tau^{a})\right\|_{L^{2}} dt^{\prime} \\\nonumber
\lesssim & \int_{0}^{t}(t-t^{\prime})^{-\frac{2\beta-1}{2\beta^{2}}}\varepsilon^{2} a^{\frac{1}{2\beta}}(1+t^{\prime})^{-\frac{1}{\beta}}+(t-t^{\prime})^{-\frac{1}{2\beta}}\varepsilon^{2} a^{\frac{1}{2\beta}}(1+t^{\prime})^{-\frac{6\beta^2-4\beta+1}{4\beta^{2}}}dt^{\prime} \\\nonumber
\lesssim &\ \varepsilon^{2} a^{\frac{1}{2\beta}}.
 \end{align*}
Similarly, we find 
 \begin{align*}
\left\|\int_0^te^{-(-\Delta)^{\beta}(t-t^{\prime})}[Q_2(\nabla u^0,\tau^0)-Q_2(\nabla u^a,\tau^a)]dt^{\prime}\right\|_{L^2}\lesssim \ \varepsilon^{2} a^{\frac{1}{2\beta}}.
 \end{align*}
Combining the above estimates, it comes out
\begin{align*}
\|\tilde{\tau}^{a}-\tilde{\tau}^{0}\|_{L^{\infty}([0,\infty),L^{2})}\ge Ca^{ \frac{1}{2\beta}}.
\end{align*}
\end{proof}

\smallskip
\noindent\textbf{Acknowledgments} This work was partially supported by the National Natural Science Foundation of
 China (No.12171493).
%The authors thank the referee for valuable comments and suggestions.

\phantomsection
\addcontentsline{toc}{section}{\refname}
%添加参考文献到书签，宏包 hyperref
\bibliographystyle{abbrv} %plain ,%alpha, %abbrv
\bibliography{Oldroyd-Bref}

\end{document}